\documentclass[a4paper,10pt]{amsart} 

\usepackage{verbatim, amssymb, enumerate}

\renewcommand \a{\alpha}
\renewcommand \b{\beta}
\newcommand \K{\delta}
\newcommand \n{\nabla}
\newcommand \la{\lambda}
\newcommand \ve{\varepsilon}
\newcommand \id{\mathrm{id}}
\newcommand \br{\mathbb{R}}
\newcommand \bc{\mathbb{C}}
\newcommand \Oc{\mathbb{O}}
\newcommand \Rn{\mathbb R^n}
\newcommand \Ro{\mathbb R^8}
\newcommand \Rs{\mathbb R^{16}}

\newcommand \rk{\mathrm{rk}}
\newcommand \Ker{\mathrm{Ker}}
\newcommand \End{\mathrm{End}}

\newcommand \Cliff{\mathrm{Cliff}}
\newcommand \Cl{\mathrm{Cl}}
\newcommand \Span{\mathrm{Span}}
\newcommand \Tr{\mathrm{Tr}}

\newcommand \cJ{\mathcal{J}}
\newcommand \cI{\mathcal{I}}
\newcommand \cp{\mathcal{C}}
\newcommand \og{\mathfrak{o}}

\renewcommand \O{\mathrm{O}}
\newcommand \OC{\Oc \otimes \bc}
\newcommand \<{\langle}
\renewcommand \>{\rangle}
\newcommand \JC{\mathcal{J}_{\bc}}
\newcommand \mU{\mathcal{U}}

\theoremstyle{plane}
\newtheorem{theorem}{Theorem}
\newtheorem*{theorem*}{Theorem}
\newtheorem*{corollary*}{Corollary}
\newtheorem{lemma}{Lemma}

\theoremstyle{definition}
\newtheorem{definition}{Definition}

\theoremstyle{remark}
\newtheorem{remark}{Remark}

\begin{document}

\title{Conformally Osserman manifolds}

\author{Y.Nikolayevsky}
\address{Department of Mathematics and Statistics, La Trobe University, Victoria, 3086, Australia}
\email{y.nikolayevsky@latrobe.edu.au}

\date{\today}

\dedicatory{
Dedicated to the memory of Novica Bla\v zi\'c (1959 -- 2005), a remarkable mathematician 
and a wonderful person.}
\thanks{Supported by the FSTE grant} 


\begin{abstract}
An algebraic curvature tensor is called Osserman if the eigenvalues of the associated Jacobi operator are constant on
the unit sphere. A Riemannian manifold is called conformally Osserman if its Weyl conformal curvature tensor at every
point is Osserman. We prove that a conformally Osserman manifold of dimension $n \ne 3, 4, 16$ is locally conformally
equivalent either to a Euclidean space or to a rank-one symmetric space.
\end{abstract}

\maketitle

\section{Introduction}
\label{s:intro}

An \emph{algebraic curvature tensor} $\mathcal{R}$ on a Euclidean space $\Rn$ is a $(3, 1)$ tensor having the same
symmetries as the curvature tensor of a Riemannian manifold. For $X \in \Rn$, the \emph{Jacobi operator}
$\mathcal{R}_X : \Rn \to \Rn$ is defined by $\mathcal{R}_XY = \mathcal{R}(X, Y)X$ . The Jacobi operator is symmetric
and $\mathcal{R}_XX = 0$ for all $X \in \Rn$.

\begin{definition} \label{d:oact}
An algebraic curvature tensor $\mathcal{R}$ is called \emph{Osserman} if the eigenvalues of the
Jacobi operator $\mathcal{R}_X$ do not depend on the choice of a unit vector $X \in \Rn$.
\end{definition}

One of the algebraic curvature tensors naturally associated to a Riemannian manifold (apart from the curvature tensor
itself) is the Weyl conformal curvature tensor.

\begin{definition} \label{d:com}
A Riemannian manifold is called \emph{(pointwise) Osserman} if its curvature tensor at every
point is Osserman. A Riemannian manifold is called \emph{conformally Osserman} if its Weyl tensor at every
point is Osserman.
\end{definition}

It is well-known (and is easy to check directly) that a Riemannian space locally isometric to a Euclidean space or to
a rank-one symmetric space is Osserman. The question of whether the converse is true (``every pointwise Osserman
manifold is flat or locally rank-one symmetric") is known as the \emph{Osserman Conjecture} \cite{O}. The first result
on the Osserman Conjecture (the affirmative answer for manifolds of dimension not divisible by $4$) was published
before the conjecture itself \cite{Chi}. In the following almost two decades, the research in the area of Osserman and
related classes of manifolds, both in the Riemannian and pseudo-Riemannian settings, was flourishing, with dozens of
papers and at least three monographs having been published \cite{G1,G2,GKV}.

At present, the Osserman Conjecture is proved almost completely, with the only exception when the
dimension of an Osserman manifold is $16$ and one of the eigenvalues of the Jacobi operator has multiplicity $7$ or
$8$ \cite{Nhjm,Nmm,Nma,Nbel}.
The main difficulty lies in the fact that the Cayley projective plane (and its hyperbolic dual)
are Osserman, with the multiplicities of the eigenvalues of the Jacobi operator being exactly $7$ and $8$; moreover,
the curvature tensor of the Cayley projective plane is \emph{essentially} different from that
of the other rank-one symmetric spaces, as it does not admit a Clifford structure (see Section~\ref{s:clifford}
for details). This is the only known Osserman curvature tensor without a Clifford structure, and to prove the
Osserman Conjecture in full it would be largely sufficient to show that there are no other exceptions.

The study of conformally Osserman manifolds was started in \cite{BG1}, and then continued in 
\cite{BG2,BGNSi,G2,BGNSt}.
Every Osserman manifold is conformally Osserman (which easily follows from the formula for the Weyl tensor and the
fact that every Osserman manifold is Einstein), as also is every manifold locally conformally equivalent to
an Osserman manifold.

Our main results is the following theorem.
\begin{theorem} \label{t:co}
A connected $C^{\infty}$ Riemannian conformally Osserman manifold of dimension $n \ne 3,4,16$ is locally conformally
equivalent to a Euclidean space or to a rank-one symmetric space.
\end{theorem}


Theorem~\ref{t:cocl} answers, with three exceptions, the conjecture made in \cite{BGNSi} (for conformally Osserman
manifolds of dimension $n > 6$ not divisible by $4$, this conjecture is proved in \cite[Theorem~1.4]{BG1}).

Note that the nature of the three excepted dimensions in Theorem~\ref{t:cocl} is different. In dimension three the
Weyl tensor
gives no information on a manifold at all. In dimension four, even a ``genuine" pointwise Osserman manifold does
not have to be locally symmetric (see \cite[Corollary 2.7]{GSV}, \cite{Ol},
for the examples of ``generalized complex space forms"). As it is proved in \cite{Chi}, the Osserman Conjecture is
still true in dimension four, but in a more restrictive version: one requires the eigenvalues of the Jacobi operator
to be constant on the whole unit tangent bundle (a Riemannian manifold with this property is called
\emph{globally Osserman}). One might wonder, whether the conformal counterpart of this result is true. The following
elegant characterization in dimension four is obtained in \cite{BG2}: a four-dimensional Riemannian manifold is
conformally Osserman if and only if it is either self-dual or anti-self-dual.

In dimension $16$, both the conformal and the original Osserman Conjecture remain open (for partial
results, see \cite{Nma,Nbel} in the Riemannian case and Theorem~\ref{t:cocl} in Section~\ref{s:conf} in the
conformal case).

As a rather particular case of Theorem~\ref{t:co}, we obtain the following analogue of the Weyl-Schouten Theorem for
rank-one symmetric spaces: a Riemannian manifold of dimension greater than four having ``the same" Weyl tensor as
that of one of the complex/quaternionic projective spaces or their noncompact duals is locally conformally equivalent
to that space. More precisely:

\begin{theorem}\label{t:corollary}
Let $M_0^n$ denote one of the spaces $\bc P^{n/2}, \; \bc H^{n/2}, \; \mathbb{H}P^{n/4}, \; \mathbb{H}H^{n/4}$, and
let $W_0$ be the Weyl tensor of $M_0^n$ at some point $x_0 \in M_0^n$. Suppose that for every point $x$ of a
Riemannian manifold $M^n, \; n > 4$, there exists a linear isometry $\iota: T_xM^n \to T_{x_0}M_0^n$ which maps the
Weyl tensor of $M^n$ at $x$ on a positive multiple of $W_0$. Then $M^n$ is locally conformally equivalent to $M_0^n$.
\end{theorem}


For $M_0^n= \bc P^{n/2}, \; \bc H^{n/2}$ and $n >6$, the claim follows from \cite[Theorem~1.4]{BG1}. The fact that
the dimension $n=16$ is not excluded (compared to Theorem~\ref{t:co}) follows from Theorem~\ref{t:cocl} (see
Section~\ref{s:conf}).

We explicitly require all the object (manifolds, metrics, vector and tensor fields) to be smooth (of class $C^\infty$),
although all the results remain valid for class $C^k$, with sufficiently large $k$.

\smallskip

The paper is organized as follows. In Section~\ref{s:clifford}, we give a background on Osserman algebraic
curvature tensors and on Clifford structures and prove some technical Lemmas. 
The proof of Theorem~\ref{t:co} is given in Section~\ref{s:conf}. Theorem~\ref{t:co} is deduced from a more general
Theorem~\ref{t:cocl}. We first prove the local version using the differential Bianchi identity, and then the global
version by showing that the ``algebraic type" of the Weyl tensor is the same at all the points of a connected
conformally Osserman Riemannian manifold (in particular, a nonzero Osserman Weyl tensor cannot degenerate to zero).

\section{Algebraic curvature tensors with a Clifford structure}
\label{s:clifford}

\subsection{Clifford structure}
\label{ss:clifford}

The property of an algebraic curvature tensor $\mathcal{R}$ to be Osserman is quite algebraically restrictive. In the
most cases, such a tensor can be obtained by the following remarkable construction, suggested in \cite{GSV}, which
generalizes the curvature tensors of the complex and the quaternionic projective spaces.

\begin{definition}\label{d:cl}
A \emph{Clifford structure} $\Cliff(\nu ; J_1, \dots, J_\nu; \la_0, \eta_1, \dots , \eta_\nu)$ on a
Euclidean space $\Rn$ is a set of $\nu \ge 0$ anticommuting almost Hermitian structures $J_i$ and $\nu+1$ real numbers
$\la_0, \eta_1, \dots \eta_\nu$, with $\eta_i \ne 0$. An algebraic curvature tensor $\mathcal{R}$ on $\Rn$ \emph{has a
Clifford structure} $\Cliff(\nu ; J_1, \dots, J_\nu; \la_0, \eta_1, \dots , \eta_\nu)$ if
\begin{equation}\label{eq:cs}
\mathcal{R}(X, Y) Z = \la_0 (\< X, Z \> Y - \< Y, Z \> X) \\
+ \sum\nolimits_{i=1}^\nu \eta_i (2 \< J_iX, Y \> J_iZ + \< J_iZ, Y \> J_iX - \< J_iZ, X \> J_iY).
\end{equation}
When it does not create ambiguity, we abbreviate
$\Cliff(\nu ; J_1, \dots, J_\nu; \la_0, \eta_1, \dots , \eta_\nu)$ to just $\Cliff(\nu)$.
\end{definition}

\begin{remark}\label{rem:climpliesos}
As it follows from Definition~\ref{d:cl}, the operators $J_i$ are skew-symmetric, orthogonal and satisfy the equations
$\<J_iX, J_jX\> = \K_{ij}\|X\|^2$ and $J_iJ_j+J_jJ_i =-2 \delta_{ij} \id$, for all $i,j =1, \dots, \nu$, and all
$X \in \Rn$. This implies that every algebraic curvature tensor with a Clifford structure is Osserman, as
by \eqref{eq:cs} the Jacobi operator has the form
$\mathcal{R}_XY =  \la_0 (\| X\|^2 Y - \< Y, X \> X) + \sum\nolimits_{i=1}^\nu 3 \eta_i \< J_iX, Y \> J_iX$, so
for a unit vector $X$, the eigenvalues of $\mathcal{R}_X$ are $\la_0$ (of multiplicity $n-1-\nu$ provided $\nu<n-1$),
$0$ and $\la_0+3\eta_i, \; i=1, \dots, \nu$.
\end{remark}

The converse (``every Osserman algebraic curvature tensor has a Clifford structure") is true in all the dimensions
except for $n=16$, and also in many cases when $n=16$, as follows from \cite{Nma} (Proposition~1 and the second last
paragraph of the proof of Theorem~1 and Theorem~2), \cite[Proposition~1]{Nmm} and \cite[Proposition~2.1]{Nbel}.
The only known counterexample is the curvature tensor $R^{\O P^2}$ of the Cayley projective plane (more precisely,
any algebraic curvature tensor of the form
$\mathcal{R}= a R^{\O P^2}+b R^1$, where $R^1$ is the curvature tensor of the unit sphere $S^{16}(1)$ and $a \ne 0$).

A Clifford structure $\Cliff(\nu)$ on the Euclidean space $\Rn$ turns it the into a Clifford module (we refer to
\cite[Part~1]{ABS}, \cite[Chapter~11]{H}, \cite[Chapter~1]{LM} for standard facts on Clifford algebras and
Clifford modules). Denote $\Cl(\nu)$ a \emph{Clifford algebra} on $\nu$ generators  $x_1, \dots, x_\nu$, an
associative unital algebra over $\br$ defined by the relations $x_ix_j+x_jx_i=-2 \K_{ij}$ (this condition determines
$\Cl(\nu)$ uniquely). The map $\sigma: \Cl(\nu) \to \Rn$ defined on generators by $\sigma(x_i) = J_i$ (and
$\sigma(1) = \id$) is a representation of $\Cl(\nu)$ on $\Rn$. As all the $J_i$'s are orthogonal and skew-symmetric,
$\sigma$ gives rise to an \emph{orthogonal multiplication} defined as follows. In the Euclidean space $\br^{\nu}$,
fix an orthonormal basis $e_1, \dots, e_\nu$. For every $u = \sum_{i=1}^{\nu} u_i e_i \in \br^{\nu}$ and every
$X \in \Rn$, define
\begin{equation}\label{eq:ortmult}
    J_uX = \sum\nolimits_{i=1}^{\nu} u_i J_iX
\end{equation}
(when $u=e_i$, we abbreviate $J_{e_i}$ to $J_i$). The map $J: \br^{\nu} \times \Rn \to \Rn$ defined by
\eqref{eq:ortmult} is an orthogonal multiplication: $\|J_uX\|^2=\|u\|^2\|X\|^2$ (similarly, we can define
an orthogonal multiplication $J: \br^{\nu+1} \times \Rn \to \Rn$ by
$J_uX = u_0 X + \sum\nolimits_{i=1}^{\nu} u_i J_iX$, for $u = \sum_{i=0}^{\nu} u_i e_i \in \br^{\nu+1}$, where
$e_0, e_1, \dots, e_\nu$ is an orthonormal basis for the Euclidean space $\br^{\nu+1}$).
For $X \in \Rn$, denote
\begin{equation*}
    \cJ X = \Span(J_1X, \dots, J_\nu X), \qquad \cI X = \Span(X, J_1X, \dots, J_\nu X).
\end{equation*}
Later we will also use the complexified versions of these subspaces which we denote
$\JC X$ and ${\mathcal{I}_{\bc}} X$ respectively, for $X \in \bc^n$.

If $\Rn$ is a $\Cl(\nu)$-module (equivalently, if there exists an algebraic curvature tensor with a
Clifford structure $\Cliff(\nu)$ on $\Rn$), then (see, for instance, \cite[Theorem~11.8.2]{H})
\begin{equation}\label{eq:radon}
\nu \le 2^b + 8 a-1, \qquad \text{where $n= 2^{4a+b} c, \; c$ is odd, $0 \le b \le 3$}.
\end{equation}
From \eqref{eq:radon}, we have the following inequalities.
\begin{lemma} \label{l:nvsnu}
Let $\mathcal{R}$ be an algebraic curvature tensor with a Clifford structure $\Cliff(\nu)$ on
$\Rn$. Suppose $n \ne 2,4,8,16$. Then
\begin{enumerate}[\rm (i)]
  \item $n \ge 3 \nu+3$, with the equality only when $n=6, \; \nu=1$, or $n = 12, \; \nu=3$, or $n=24, \; \nu=7$.
  \item $n > 4\nu-2$, except in the following cases: $n=24, \; \nu=7$ and $n=32, \; \nu=8$.
  \item there exists an integer $l$ such that $\nu < 2^l < n$.
\end{enumerate}
\end{lemma}

\subsection{Clifford structures on $\Ro$ and the octonions}
\label{ss:cliffoct}

The proof of Theorem~\ref{t:co} in the ``generic case" will rely upon the fact that $\nu$ is small relative to $n$
(with the required estimates given in Lemma~\ref{l:nvsnu}). However, in the case $n=8$, the number $\nu$ can be as
large as $7$, according to \eqref{eq:radon}. Consider this case in more detail.
As it is shown in \cite{Nmm}, not only every Osserman algebraic curvature tensor $\mathcal{R}$ on $\Ro$ has a Clifford
structure, but also that Clifford can be taken of one of the two (mutually exclusive) forms: either $\mathcal{R}$ has
a $\Cliff(3)$-structure, with $J_1J_2=J_3$, or an existing $\Cliff(\nu)$-structure can be ``complemented" to a
$\Cliff(7)$-structure. More precisely:

\begin{lemma}\label{l:formofR}
{\ }
\begin{enumerate}[1.]
  \item
  Suppose $\mathcal{R}$ is an algebraic curvature tensor on $\Ro$ having a Clifford structure
  $\Cliff(\nu ; J_1, \dots, J_\nu;$ $\la_0, \eta_1, \dots , \eta_\nu)$. Then exactly one of the following two
  possibilities may occur: either $\mathcal{R}$ has a Clifford structure $\Cliff(3)$ with $J_1J_2 = J_3$,
  or there exist $7-\nu$ operators $J_{\nu+1}, \dots , J_7$ such that $J_1, \dots, J_7$
  are anticommuting almost Hermitian structures with $J_1J_2\ldots J_7 = \id_{\Ro}$ and $\mathcal{R}$ has a Clifford
  structure $\Cliff(7 ; J_1, \dots, J_7; \la_0-3\xi, \eta_1+\xi, \dots , \eta_\nu+\xi, \xi, \dots, \xi)$, for any
  $\xi \ne -\eta_i, 0$.
  \item
  Let $\Oc$ be the octonion algebra with the inner product defined by $\|u\|^2=uu^*$, where ${}^*$ is the octonion
  conjugation, and let $\Oc'=1^\perp$, the space of imaginary octonions. Then, in the second case in assertion 1,
  there exist linear isometries
  $\iota_1: \Ro \to \Oc, \; \iota_2: \br^7 \to \Oc'$ such that the orthogonal multiplication \eqref{eq:ortmult} is
  given by $J_uX = \iota_1(X) \iota_2(u)$.
\end{enumerate}
\end{lemma}
\begin{proof}
1. This assertion is proved in \cite[Lemma~5]{Nmm}. The proof is based on the fact that every representation
$\sigma$ of $\Cl(\nu)$ on $\Ro$, except for the representations of $\Cl(3)$ with $J_1J_2=\pm J_3$, is a restriction of
a representation of $\Cl(7)$ on $\Ro$, to $\Cl(\nu) \subset \Cl(7)$.
It follows that the almost Hermitian structures $J_1, \dots , J_\nu$ defined by $\sigma$ can be complemented
by almost Hermitian structures $J_{\nu+1}, \dots , J_7$ such that $J_1, \dots , J_7$ anticommute, and so
$\mathcal{R}$ can be written in the form \eqref{eq:cs}, with a formal summation up to $7$ on the right-hand side
(but with $\eta_i=0$ when $i=\nu+1, \dots, 7$). To obtain a $\Cliff(7)$-structure for $\mathcal{R}$, according to
Definition~\ref{d:cl}, we only need to make all the $\eta_i$'s nonzero. This can be done using the identity
\begin{equation}\label{eq:8sq}
\< X, Z \> Y - \< Y, Z \> X = \sum\nolimits_{i=1}^7 \tfrac 13
(2 \< J_iX, Y \> J_iZ + \< J_iZ, Y \> J_iX - \< J_iZ, X \> J_iY)
\end{equation}
(which is obtained from the polarized identity $\|X\|^2 Y - \<X, Y\> X = \sum_{i=1}^7\< J_iX, Y \> J_iX$ which
follows from the fact that, for $X \ne 0$, the vectors $\|X\|^{-1}X, \|X\|^{-1}J_1X, \dots, \|X\|^{-1}J_7X$ form
an orthonormal basis for $\Ro$). Then by \eqref{eq:cs}, $\mathcal{R}$ has a Clifford structure
$\Cliff(7 ; J_1, \dots, J_7; \la_0-3\xi, \eta_1+\xi, \dots , \eta_\nu+\xi, \xi, \dots, \xi)$, for any
$\xi \ne -\eta_i, 0$.

2. This assertion is also proved in \cite{Nmm} (see the beginning of Section~5.1). The proof is based on the
following. There are two nonisomorphic representations of $\Cl(7)$ on $\Ro$. Identifying $\Ro$ with the octonion
algebra $\Oc$ via a linear isometry these representations are given by the orthogonal multiplications $J_uX=uX$ and
$J_uX=Xu$ respectively \cite[\S~I.8]{LM}. As $(uX)^*=X^*u^*=-X^*u$ for all
$u, X \in \Oc, \; u \perp 1$, the first representation is orthogonally equivalent to the second one, with the
operators $J_i$ replaced by $-J_i$. Since changing the signs of the $J_i$'s does not affect the form of the
algebraic curvature tensor \eqref{eq:cs}, we can always assume that a $\Cliff(7)$-structure for an algebraic
curvature tensor on $\Ro$ is given by the orthogonal multiplication $J_uX =\iota_1(X) \iota_2(u)$.
\end{proof}

In the proof of Theorem~\ref{t:co} for $n=8$, we will usually identify $\Ro$ with $\Oc$ and of $\br^7$ with $\Oc'$ via
some fixed linear isometries $\iota_1, \iota_2$ and simply write the orthogonal multiplication in the form
\begin{equation}\label{eq:JuX8}
J_uX = X u,
\end{equation}
where $X \in \Ro = \Oc, \; u \in \Oc'$.
The proof of Theorem~\ref{t:co} for $n=8$ extensively uses the computations in the octonion algebra $\Oc$
(in particular, the standard identities like $a^*=2 \<a,1\> 1 - a, \; \<a,b\>=\<a^*, b^*\>=\frac12(a^*b+b^*a), \;
a(ab)=a^2b, \; \<a, bc\> = \<b^*a, c\> = \<ac^*, b\>, \; (ab^*)c+(ac^*)b = 2\<b, c\>a, \;
\<ab,ac\>=\<ba,ca\>=\|a\|^2 \<b,c\>$, for any $a, b, c \in \Oc$,
and the similar ones, see e.g. \cite[Section~IV]{HL}) and
the fact that $\Oc$ is a division algebra (in particular, any nonzero octonion is invertible:
$a^{-1} = \|a\|^{-2} a^*$). We will also use the \emph{bioctonions} $\Oc \otimes \bc$, the algebra over the $\bc$ with
the same multiplication table as that for $\Oc$. As all the above identities are polynomial, they still hold for
bioctonions, with the complex inner product on $\bc^8$, the underlying linear space of $\Oc \otimes \bc$. However,
the bioctonion algebra is not a division algebra (and has zero-divisors: $(i1+e_1)(i1-e_1)=0$).

\subsection{Technical lemma}
\label{ss:lemma}

In the proof of Theorem~\ref{t:co}, we will use the following lemma.

\begin{lemma}\label{l:level}
Suppose that $n > 4$, and additionally, if $n=8$, then $\nu \le 3$, and if $n=16$, then $\nu \le 7$.

1. Let $F:\Rn \to \Rn$ be a homogeneous polynomial map of degree $m$ such that for all $X \in \Rn$, $F(X) \in \cJ X$
(respectively $F(X) \in \cI X$). Then there exist homogeneous polynomials $c_i, \; i=1, \dots, \nu$
(respectively $i=0, 1, \dots, \nu$), of degree $m-1$ such that
$F(X)=\sum_{i=1}^\nu c_i(X) J_iX$ (respectively $F(X)=c_0(X) X + \sum_{i=1}^\nu c_i(X) J_iX$).

2.
Let $1 \le k \le \nu$ and let $a_j, \; 1 \le j \le \nu, \, j \ne k$, be $\nu-1$ vectors in $\Rn$ such
that for all $Y \in \Rn$,
\begin{equation}\label{eq:JJk}
\sum\nolimits_{j \ne k} (\<a_j, J_kY\> J_jY + \<a_j, Y\> J_kJ_jY)=0.
\end{equation}
Then either $a_j=0$ for all $j \ne k$, or $\nu=1$, or $\nu=3, \; J_1J_2 =\ve J_3, \; \ve = \pm 1$,
and $a_j= J_j v$ for all $j \ne k$, where $v \ne 0$.

3. Suppose $n$ and $\nu$ are arbitrary numbers satisfying \eqref{eq:radon}.
Let $N^n$ be a smooth Riemannian manifold and let $J_1,\dots,J_\nu$ be anticommuting almost Hermitian
structures on $N^n$. Suppose that for every nowhere vanishing smooth vector field $X$ on $N^n$, the distribution
$\cJ X=\Span(J_1X,\dots, J_\nu X)$ is smooth (that is, the $\nu$-form $J_1X \wedge \dots \wedge J_\nu X$ is
smooth). Then for every $x \in N^n$, there exists a neighbourhood $\mathcal{U}=\mathcal{U}(x)$ and
smooth anticommuting almost Hermitian structures $\tilde J_1,\dots,\tilde J_\nu$ on $\mathcal{U}$ such that
$\Span(\tilde J_1X, \dots, \tilde J_\nu X)=\Span(J_1X,\dots, J_\nu X)$, for any vector field $X$ on $\mathcal{U}$.
\end{lemma}

\begin{proof}
1. It is sufficient to prove the assertion for the case $F(X) \in \cI X$.

As for every $X \ne 0$, the vectors $X, J_1X, \dots, J_\nu X$ are orthogonal and have the same length $\|X\|$, we
have $\|X\|^2F(X) = f_0(X)X + \sum_{i=1}^\nu f_i(X) J_iX$, where $f_0(X) = \<F(X),X\>, \; f_i(X) = \<F(X),J_iX\>$
are homogeneous polynomials of degree $m+1$ of $X$ (or possibly zeros). Taking the squared lengths of the both sides
we get $\|X\|^2 \|F(X)\|^2 = f_0^2(X) + \sum_{i=1}^\nu f_i^2(X)$, so the sum of squares of $\nu+1$ polynomials
$f_0(X), f_1(X), \dots, f_\nu(X)$ is divisible by $\|X\|^2$. Let for $X = (x_1, \dots, x_n), \; (\|X\|^2)$ be the
ideal of $\br[X]$ generated by $\|X\|^2 = \sum_j x_j^2$, and let $\mathbf{R} = \br[X]/(\|X\|^2)$.
We have $\sum_{i=0}^\nu \hat{f}_i^2=0$, where $\hat{f}_i$ is the image of $f_i$ under the natural projection
$\pi:\br[X] \to \mathbf{R}$. If at least one of the $\hat{f}_i$'s is nonzero (say the $\nu$-th one), then
$\sum_{i=0}^{\nu-1} (\hat{f}_i/\hat{f}_\nu)^2=-1$ in $\mathbb{F}$, the field of fractions of the ring $\mathbf{R}$.
The field $\mathbb{F}$ is isomorphic to the field $\mathbb{L}_{n-1} = \br(x_1, \ldots, x_{n-1}, \sqrt{-d})$, where
$d=x_1^2 + \ldots + x_{n-1}^2$ (an isomorphism from $\mathbb{L}_{n-1}$ to $\mathbb{F}$ is induced by the map
$(a + b \sqrt{-d})/c \to (a + bx_n)/c$, with $a, b, c \in \br[x_1, \dots, x_{n-1}], \; c \ne 0$). By
\cite[Theorem 3.1.4]{Pf}, the \emph{level} of the field $\mathbb{L}_{n-1}$, the minimal number of elements whose
sum of squares is $-1$, is $2^l$, where $2^l < n \le 2^{l+1}$. It follows that in all the cases when
$\nu < 2^l <n$ we arrive at a contradiction. This means that $\hat{f}_i=0$, for all $i=0, \dots, \nu$, so each
of the $f_i$'s is divisible by $\|X\|^2$ in $\br[X]$, so
$F(X) = (\|X\|^{-2}f_0(X))X + \sum_{i=1}^\nu (\|X\|^{-2}f_i(X)) J_iX$, with all the nonzero coefficients on the
right-hand side being homogeneous polynomials of degree $m-1$. The claim now follows from assertion~(iii) of
Lemma~\ref{l:nvsnu}.

2. If $\nu=1$, equation \eqref{eq:JJk} is trivially satisfied. If $\nu=2$, the claim immediately follows
by taking the inner product of \eqref{eq:JJk} with $J_1J_2Y$. If $\nu=3$, let $k=3$ (without loss of generality).
Taking the inner product of \eqref{eq:JJk} with $J_1Y$ we obtain
$\<a_1, J_3Y\> \|Y\|^2 = \<a_j, Y\> \<J_1J_3J_2Y,Y\>$. It follows that the polynomial $\<J_1J_3J_2Y,Y\>$ is divisible
by $\|Y\|^2$. As the operator $J_1J_3J_2$ is symmetric and orthogonal, it equals $\pm \id$. Hence
$J_1J_2=\varepsilon J_3, \; \varepsilon=\pm1$. Then \eqref{eq:JJk} takes the form
$\<a_1, J_3Y\> J_1Y + \<a_2, J_3Y\> J_2Y +\varepsilon\<a_1, Y\> J_2Y-\varepsilon\<a_2, Y\> J_1Y=0$, which is equivalent
to $a_1=-\varepsilon J_3a_2$. Acting by $J_1$ on the both sides we obtain $J_1a_1=J_2a_2$, so $a_j=J_jv$, with
$v=-J_1a_1=-J_2a_2$ (we can assume $v \ne 0$, as otherwise $a_j=0$).

Now assume $\nu > 3$ and denote $L = \Span(a_j)$. As it follows from \eqref{eq:JJk}, if $Y \perp L$,
then $J_kY \perp L$, so $L$ is $J_k$-invariant. Polarizing \eqref{eq:JJk} we obtain
$\sum\nolimits_{j \ne k} (\<a_j, J_kX\> J_jY + \<a_j, X\> J_kJ_jY)+
\sum\nolimits_{j \ne k} (\<a_j, J_kY\> J_jX + \<a_j, Y\> J_kJ_jX)=0$. It follows that for all $X \perp L$ and all
$Y \in \Rn$, $\sum\nolimits_{j \ne k} (\<a_j, J_kY\> J_jX + \<a_j, Y\> J_kJ_jX)=0$, that is,
$J_{u(Y)}X= J_kJ_{v(Y)}X$, where $u(Y) =\sum_{j \ne k} \<a_j, J_kY\> e_j, \; v(Y)=\sum_{j \ne k} \<a_j, Y\> e_j$. Note
that $u(Y), v(Y) \perp e_k$. Now, fix an arbitrary $Y \in \Rn$ and choose a unit vector $w \perp u(Y),v(Y),e_k$ (this
is possible, as $\nu >3$). Then $J_wJ_{u(Y)}X= J_wJ_kJ_{v(Y)}X$, so $\<J_wJ_kJ_{v(Y)}X,X\>=0$, for all
$X \in L^\perp$. If $v(Y) \ne 0$, then the operator $\|v(Y)\|^{-1}J_wJ_kJ_{v(Y)}$ is symmetric and orthogonal, so the
maximal dimension of its isotropic subspace is $\frac12 n < n-(\nu-1)=\dim L^\perp$ (the inequality follows from
assertion~(ii) of Lemma~\ref{l:nvsnu}), which is a contradiction. Hence $v(Y)=0$ for all $Y \in \Rn$, so all the
$a_j$'s are zeros.

3. We first prove the lemma assuming $2\nu \le n$. In this case, the proof closely follows the arguments of the proof
of \cite[Lemma~3.1]{Nhjm}.

Let $Y_0 \in T_xN^n$ be a unit vector. As $2\nu \le n$, there exists a unit vector $E \in T_x N^n$ which is not in
the range of the map $\Phi: S^{\nu-1} \times S^{\nu-1} \to S^{n-1}$ defined by $\Phi(u,v)= J_uJ_vY_0$. Then
$\cJ E \cap \cJ Y_0=0$. It follows that on some neighbourhood $\mathcal{U}'$ of $x$ there exist smooth unit
vector fields $Y$ and $E_n$ such that $E_n(x)=E, \; Y(x)=Y_0$ and $\cJ E_n \cap \cJ Y=0$ at every point $y \in \mU'$.
By the assumption, the $\nu$-dimensional distribution $\cJ E_n$ is smooth, so we can choose $\nu$
smooth orthonormal sections $E_1, \dots, E_\nu$ of it, and then define
anticommuting almost Hermitian structures $\tilde J_\a$ on $\mU'$ by $\tilde J_\a E_n = E_\a$ (so that
$\tilde J_\a= \sum_{\b=1}^\nu a_{\a\b}J_\b$, where $(a_{\a\b})$ is the $\nu \times \nu$ orthogonal matrix given by
$a_{\a\b}=\<E_\a,J_\b E_n\>$).

Let $E_{\nu+1}, \dots, E_{n-1}$ be orthonormal vector fields on $\mU'$ such that $E_1, \dots, E_n$ is an orthonormal
frame, and let, for a vector field $X$ on $\mU'$, $\tilde J X$ denote the $n \times \nu$ matrix whose column vectors
are $\tilde J_1 X, \dots, \tilde J_\nu X$ relative to the frame $E_1, \dots, E_n$. Then
$(\tilde J X)^t \tilde J X = \|X\|^2 I_\nu$ and all the $\nu \times \nu$ minors of the matrix $\tilde J X$ are
smooth functions on $\mU'$. Moreover, the entries of the matrices $\tilde J E_i, \; i=1, \dots, n$, are the
rearranged entries of the matrices $\tilde J_\a, \a=1, \dots, \nu$, relative to the basis $\{E_i\}$, so to prove
that the $\tilde J_\a$'s are smooth it suffices to show that all the entries of the matrices
$\tilde J E_i$ are smooth (on a possibly smaller neighbourhood). Denote
$\tilde J E_i=\left(\begin{smallmatrix} K_i \\ P_i \end{smallmatrix}\right)$, where $K_i$ and $P_i$ are
$\nu \times \nu$ and $(n-\nu) \times \nu$ matrices-functions on $\mU'$ respectively
(note that $\tilde J E_n=\left(\begin{smallmatrix} I_\nu \\ 0 \end{smallmatrix}\right)$). For an arbitrary $t\in \br$,
all the $\nu \times \nu$ minors of the matrix
$\tilde J (E_i+tE_n)=\left(\begin{smallmatrix} K_i+tI_\nu \\ P_i \end{smallmatrix}\right)$ are smooth.
For every entry $(P_i)_{k\a}, \; k=\nu+1, \dots, n, \; \a=1, \dots, \nu$, the coefficient of $t^{\nu-1}$ in the
$\nu \times \nu$ minor of $\tilde J (E_i+tE_n)$ consisting of $\nu-1$ out of the first $\nu$ rows (omitting the
$\a$-th row) and the $k$-th row is $\pm (P_i)_{k\a}$, so all the entries of all the $P_i$'s are smooth.

For the vector field $Y$, constructed at the beginning of the proof, denote
$\tilde J Y=\left(\begin{smallmatrix} K \\ P \end{smallmatrix}\right)$. As $P=\sum_{i=1}^n\<Y,E_i\>P_i$, all the
entries of $P$ are smooth on $\mU'$. Moreover, as $\cI Y \cap \cI E_n = 0$, the spans of the
vector columns of the matrices $\tilde J Y$ and
$\tilde J E_n=\left(\begin{smallmatrix} I_\nu \\ 0\end{smallmatrix}\right)$ have trivial intersection, so
$\rk\, P = \nu$, at every point $y \in \mU'$. Therefore we can choose the rows $\nu+1 \le b_1 < \dots < b_\nu \le n$
of the matrix $P$ at the point $x$ such that the corresponding minor $P_{(b)}=P_{b_1\dots b_\nu}$ is nonzero. Then the
same minor $P_{(b)}$ is nonzero on a (possibly smaller) neighbourhood $\mU \subset \mU'$ of $x$. Taking all the
$\nu \times \nu$ minors of $\tilde J Y$ consisting of $\nu-1$ out of $\nu$ rows of $P_{(b)}$ and one row of $K$ we
obtain that all the entries of $K$ are smooth on $\mU$. Moreover, for an arbitrary $t\in \br$,
all the $\nu \times \nu$ minors of the matrix
$\tilde J (tE_i+Y)=\left(\begin{smallmatrix} tK_i+K \\ tP_i+P \end{smallmatrix}\right)$ are smooth.
Computing the coefficient of $t$ in all the $\nu \times \nu$ minors of $\tilde J (tE_i+Y)$ consisting of $\nu-1$ out
of $\nu$ rows of $(tP_i+P)_{(b)}$ and one row of $tK_i+K$ and using the fact that all the entries of $K, P$ and $P_i$
are smooth on $\mU$ we obtain that all the entries of $K_i$ are also smooth on $\mU$.
Therefore all the entries of all the matrices $\tilde J E_i$ are smooth on $\mU$, hence the anticommuting
almost Hermitian structures $\tilde J_\a$ are also smooth on $\mU$.

As $\nu$ and $n$ must satisfy inequality \eqref{eq:radon} (hence the inequalities of Lemma~\ref{l:nvsnu}), the above
proof works in all the cases except for the following: $n=4, \, \nu=3$ and $n=8, \nu=5,6,7$. The case $n=4, \, \nu=3$
is easy: taking any smooth orthonormal frame $E_i$ on a neighbourhood of $x$ and defining
$\tilde J_\a= \sum_{\b=1}^3 a_{\a\b}J_\b$ (with the orthogonal $3 \times 3$ matrix $(a_{\a\b})$ given by
$a_{\a\b}=\<E_\a,J_\b E_4\>$) we obtain that all the entries of the $\tilde J_\a$ relative to the basis $E_i$ are
$\pm 1$ and $0$.

The proof in the cases $n=8, \nu=5,6,7$ is based on the fact that any set of
anticommuting almost Hermitian structures $J_1, \dots, J_\nu$ on $\Ro$, except when $\nu=3$ and $J_1J_2=\pm J_3$,
can be complemented by almost Hermitian structures $J_{\nu+1}, \dots, J_7$ to a set $J_1, \dots, J_7$ of
anticommuting almost Hermitian structures on $\Ro$ (assertion~1 of Lemma~\ref{l:formofR}).

If $n=8, \; \nu=7$, choose an arbitrary smooth almost Hermitian structure $J_7$ on some neighbourhood $\mU$ of
$x$ and complement it by anticommuting almost Hermitian structures $J_1, \dots, J_6$ at every point of $\mU$.
Then $\Span(J_1X, \dots, J_6X)=(\Span(X, J_7X))^\perp$ is a smooth distribution, for every smooth nowhere
vanishing vector field $X$ on $\mU$. This reduces the case $n=8, \; \nu=7$ to the case $n=8, \; \nu=6$.

Let $n=8, \; \nu=6$, and let $J_7$ be an almost Hermitian structure complementing $J_1, \dots, J_6$ at every point
$x \in N^n$. Using the first part of the proof (or the fact that $J_7X$ spans the one-dimensional smooth
distribution $(\Span(J_1X, \dots, J_6X) \oplus \br X)^\perp$, for every nonvanishing smooth vector field $X$)
we can assume that $J_7$ is smooth on a neighbourhood $\mU$ of $x \in N^n$. Choose a smooth
orthonormal frame $E_1, \dots, E_8$ on (a possibly smaller neighbourhood) $\mU$ such that the matrix of $J_7$ relative
to $E_i$ is $\left(\begin{smallmatrix} 0 & I_4 \\ -I_4 & 0 \end{smallmatrix}\right)$ and define the almost Hermitian
structure $\tilde J_6$ on $\mU$ by
$\tilde J_6E_2=E_1, \, \tilde J_6E_4=E_3, \, \tilde J_6E_6=-E_5, \, \tilde J_6E_8=-E_7$. Then $J_7$ and $\tilde J_6$
anticommute, hence we can complement them by almost Hermitian structures $J_1', \dots, J_5'$ on $\mU$ in such a way
that $J_1', \dots, J_5', \tilde J_6, J_7$ are anticommuting almost Hermitian structures. Moreover, as both $J_7$ and
$\tilde J_6$ are smooth on $\mU$, the five-dimensional distribution
$\Span(J_1'X,\dots, J_5'X)=(\Span(X, J_7X, \tilde J_6X))^\perp$ is smooth, for every smooth nowhere
vanishing vector field $X$ on $\mU$. This reduces the case $n=8, \; \nu=6$ to the case $n=8, \; \nu=5$. Indeed, if
$\tilde J_1, \dots, \tilde J_5$ are smooth anticommuting almost Hermitian structures on $\mU$ such that
$\Span(\tilde J_1X,\dots, \tilde J_5X)=\Span(J_1'X,\dots, J_5'X)$, for every vector field $X$, then
$\tilde J_1, \dots, \tilde J_5, \tilde J_6$ are the required almost Hermitian structures, as
$\Span(\tilde J_1X,\dots, \tilde J_6X)=\Span(J_1'X,\dots, J_5'X, \tilde J_6X)=(\Span(X, J_7X))^\perp
=\Span(J_1X,\dots, J_6X)$, for every vector field $X$ on $\mU$, and $\tilde J_6$ anticommutes with every
$\tilde J_\a, \; \a=1, \dots, 5$, since it anticommutes with every $J_\a', \; \a=1, \dots ,5$.

Let $n=8, \; \nu=5$, and let $J_6, J_7$ be anticommuting almost Hermitian structures
complementing $J_1, \dots, J_5$ at every point $x \in N^n$. As $\Span(J_6X, J_7X)=(\Span(J_1X,\dots, J_5X))^\perp$,
by the first part of the proof, we can choose such $J_6$ and $J_7$ to be smooth on a neighbourhood $\mU$
of $x \in N^n$. Choose a smooth
orthonormal frame $E_1, \dots, E_8$ on (a possibly smaller neighbourhood) $\mU$ as follows. First choose an arbitrary
smooth unit vector field $E_1$ on $\mU$. The vector fields $J_6E_1$ and $J_7E_1$ are orthonormal; set
$E_2=-J_6E_1$, $E_3=-J_7E_1$. The unit vector field $J_6J_7E_1$ is orthogonal to $E_1, J_6E_1, J_7E_1$; set
$E_4=-J_6J_7E_1$. Choose an arbitrary smooth unit section $E_5$ of the smooth distribution
$(\Span(E_1, E_2, E_3, E_4))^\perp$ on $\mU$. That distribution is both $J_6$- and $J_7$-invariant, so we can set,
similar to above, $E_6=J_6E_5, \; E_7=J_7E_5$, $E_8=-J_6J_7E_5$. Now define the almost Hermitian
structure $\tilde J_5$ on $\mU$ whose matrix relative to the frame $E_i$ is
$\left(\begin{smallmatrix} 0 & I_4 \\ -I_4 & 0 \end{smallmatrix}\right)$. Then $\tilde J_5,J_6,J_7$ are
anticommuting almost Hermitian structures on $\mU$, with $\tilde J_5 J_6 \ne \pm J_7$, hence we
can complement them by almost Hermitian structures $J_1', \dots, J_4'$ on $\mU$ in such a way
that $J_1', \dots, J_4', \tilde J_5, J_6, J_7$ are anticommuting almost Hermitian structures. Moreover, as
$\tilde J_5,J_6,J_7$ are smooth on $\mU$, the four-dimensional distribution
$\Span(J_1'X,\dots, J_4'X)=(\Span(X, \tilde J_5X, J_6X, J_7X))^\perp$ is smooth, for every smooth
nowhere vanishing vector field $X$ on $\mU$. By the first part of the proof, we can find smooth anticommuting
almost Hermitian structures $\tilde J_1, \dots, \tilde J_4$ on (a possibly smaller) neighbourhood $\mU$ such that
$\Span(\tilde J_1X,\dots, \tilde J_4X)=\Span(J_1'X,\dots, J_4'X)$, for every vector field $X$. Then
$\tilde J_1, \dots, \tilde J_4, \tilde J_5$ are the required almost Hermitian structures, as
$\Span(\tilde J_1X,\dots, \tilde J_5X)=\Span(J_1'X,\dots, J_4'X, \tilde J_5X)=(\Span(X, J_6X, J_7X))^\perp
=\Span(J_1X,\dots, J_5X)$, for every vector field $X$ on $\mU$, and $\tilde J_5$ anticommutes with every
$\tilde J_\a, \; \a=1,2,3,4$, since it anticommutes with every $J_\a', \; \a=1,2,3,4$.
\end{proof}

\section{Conformally Osserman manifolds. Proof of Theorem~\ref{t:co}}
\label{s:conf}

Let $M^n, \; n \ne 3, 4$, be a smooth conformally Osserman Riemannian manifold. If $n=2$, the manifold is
locally conformally flat, so we can assume that $n > 4$. Combining the results of \cite{Nma} (Proposition~1 and
the second last paragraph of the proof of Theorem~1 and Theorem~2), \cite[Proposition~1]{Nmm} and
\cite[Proposition~2.1]{Nbel} we obtain that the Weyl tensor of $M^n$ has a Clifford structure, for all
$n \ne 16$, and also for $n=16$ provided the Jacobi operator $W_X$ has an eigenvalue of multiplicity
at least $9$ (note that the Jacobi operator of any Osserman algebraic curvature tensor on $\Rs$ has an eigenvalue
of multiplicity at least $7$, by topological reasons). In the latter case, $W$ has a Clifford structure
$\Cliff(\nu)$, with $\nu \le 6$, at every point on $M^n$.

To prove Theorem~\ref{t:co} it therefore suffices to prove the following theorem.

\begin{theorem} \label{t:cocl}
Let $M^n$ be a connected smooth Riemannian manifold whose Weyl tensor at every point $x \in M^n$ has a
Clifford structure $\Cliff(\nu(x))$. Suppose that $n > 4$, and additionally that if $n=16$, then $\nu(x) \le 4$.
Then there exists a space $M^n_0$ from the list
$\Rn, \bc P^{n/2}, \; \bc H^{n/2}, \; \mathbb{H}P^{n/4}, \; \mathbb{H}H^{n/4}$
(the Euclidean space and the rank-one symmetric spaces with their standard metrics) such that $M^n$
is locally conformally equivalent to $M^n_0$.
\end{theorem}

Note that by Theorem~\ref{t:cocl}, every point of $M^n$ has a neighbourhood conformally equivalent to a domain of
the same ``model space". 
Also note  that Theorem~\ref{t:cocl}, in comparison to Theorem~\ref{t:co}, says something also in the case $n=16$.

\smallskip

We start with a brief informal sketch of the proof of Theorem~\ref{t:cocl}.
First of all, we show that the Clifford structure for the Weyl tensor can be chosen locally smooth
on an open, dense subset $M' \subset M^n$ (see Lemma~\ref{l:locc1} for the precise statement). To simplify the form
of the curvature tensor $R$ of $M^n$, we combine the $\la_0$-part of $W$ (from \eqref{eq:cs}) with the difference
$R-W$, so that $R$ has the form \eqref{eq:confcs} for some smooth symmetric operator field $\rho$, at every point
of $M'$. The technical core of the proof is Lemma~\ref{l:nablarho} and Lemma~\ref{l:nablaJ8} which establish various
identities for the covariant derivatives of $\rho$, the $J_i$'s and the $\eta_i$'s, using the
differential Bianchi identity for the curvature tensor of the form  \eqref{eq:confcs}. Lemma~\ref{l:nablaJ8}
treats the case $(n, \nu) = (8,7)$ and uses the octonion arithmetic, and Lemma~\ref{l:nablarho}, all the other cases
(and uses the fact that $\nu$ is small compared to $n$, see Lemma~\ref{l:nvsnu}).
It follows from the identities of Lemma~\ref{l:nablarho} and Lemma~\ref{l:nablaJ8} that, unless the Weyl
tensor vanishes, the metric on $M'$ can be locally changed to a conformal one whose curvature tensor
again has the form \eqref{eq:confcs}, but with the two additional features: firstly, all the
$\eta_i$'s are locally constant, and secondly, $\rho$ is a Codazzi tensor, that is, $(\n_X\rho) Y=(\n_Y\rho) X$.
By the result of \cite{DS}, the exterior products of the eigenspaces of a symmetric Codazzi tensor are invariant
under the curvature operator on the two-forms. Using that, we prove in Lemma~\ref{l:codazzi} that $\rho$ must be a
multiple of the identity, so, by \eqref{eq:confcs}, $M'$ is locally conformally equivalent to an Osserman manifold.
The affirmative answer to the Osserman Conjecture in the cases for $n$ and $\nu$ considered in
Theorem~\ref{t:cocl} \cite[Theorem~1.2]{Nhjm} implies that $M'$ is locally conformally equivalent to one of the
spaces listed in Theorem~\ref{t:cocl}. This proves Theorem~\ref{t:cocl} at the ``generic" points. To prove
Theorem~\ref{t:cocl} globally, we first show (using Lemma~\ref{l:mmeps}) that $M$ splits into a disjoint
union of a closed subset $M_0$, on which the Weyl tensor vanishes, and nonempty open connected subsets $M_\a$, each of
which is locally conformal to one of the rank-one symmetric spaces
$\bc P^{n/2}, \; \bc H^{n/2}, \; \mathbb{H}P^{n/4}, \; \mathbb{H}H^{n/4}$.
On every $M_\a$, the conformal factor $f$ is a well-defined positive smooth function. Assuming that there exists
at least one $M_\a$ and that $M_0 \ne \varnothing$ we show that there exists a point $x_0 \in M_0$ on the boundary of
a geodesic ball $B \subset M_\a$ such that both $f(x)$ and $\n f(x)$ tend to zero when $x \to x_0, \; x \in B$
(Lemma~\ref{l:x_0}). Then the positive function $u=f^{(n-2)/4}$ satisfies elliptic equation
\eqref{eq:laplacian} in $B$, with $\lim_{x \to x_0, x \in B}u(x)=0$, hence by the boundary point theorem, the limiting
value of the inner derivative of $u$ at $x_0$ must be positive.
This contradiction implies that either $M=M_0$ or $M=M_\a$.

\begin{proof}[Proof of Theorem~\ref{t:cocl}]
Let $M^n, \; n > 4$, be a connected smooth Riemannian manifold whose Weyl tensor at every point has a
Clifford structure.
Define the function $N: M^n \to \mathbb{N}$ as follows: for $x \in M^n$, $N(x)$ is the number of distinct eigenvalues
of the Jacobi operator $W_X$ associated to the Weyl tensor, where $X$ is an arbitrary nonzero vector from $T_xM^n$.
As the Weyl tensor is Osserman, the function $N(x)$ is well-defined. Moreover, as the set of symmetric
operators having no more than $N_0$ distinct eigenvalues is closed in the linear space of symmetric operators on
$\Rn$, the function $N(x)$ is lower semi-continuous (every subset $\{x \, : \, N(x) \le N_0\}$ is closed in $M^n$).
Let $M'$ be the set of points where the function $N(x)$ is continuous. It is easy to see that $M'$ is an open and
dense (but possibly disconnected) subset of $M^n$. The following lemma shows that the Clifford structure for the
Weyl tensor is locally smooth on every connected component of $M'$.

\begin{lemma} \label{l:locc1}
Let $M^n, \; n >4$, be a smooth Riemannian manifold whose Weyl tensor has a Clifford structure at every point.
If $n=16$, we additionally require that at every point $x \in M^{16}$, the Weyl tensor has a Clifford structure
$\Cliff(\nu(x))$ with $\nu(x) \ne 8$.

Let $M'$ be the (open, dense) subset of $M^n$ at the points of which the number of distinct eigenvalues of the Jacobi
operator associated to the Weyl tensor of $M^n$ is locally constant.
Then for every $x \in M'$, there exists a neighbourhood $\mathcal{U}=\mathcal{U}(x)$, a number $\nu \ge 0$,
smooth functions $\eta_1, \dots, \eta_{\nu}: \mathcal{U} \to \br \setminus \{0\}$, a smooth symmetric
linear operator  field $\rho$ and smooth anticommuting almost Hermitian structures
$J_i, \; i=1, \dots, \nu$, on $\mathcal{U}$ such that the curvature tensor of $M^n$ has the form
\begin{multline}\label{eq:confcs}
R(X, Y) Z = \< X, Z \> \rho Y +\<\rho X, Z\> Y - \< Y, Z \> \rho X -\<\rho Y, Z\> X \\
+ \sum\nolimits_{i=1}^\nu \eta_i (2 \< J_iX, Y \> J_iZ + \< J_iZ, Y \> J_iX - \< J_iZ, X \> J_iY),
\end{multline}
for all $y \in \mathcal{U}$ and $X, Y, Z \in T_yM^n$. Moreover, if $n=8$, then the curvature tensor has the form
\eqref{eq:confcs} either with $\nu=3$ and $J_1J_2=\pm J_3$, or with $\nu=7$, for all $y \in \mathcal{U}$.
\end{lemma}


\begin{proof}
Let $X$ be a smooth unit vector field on $M^n$. As the Weyl tensor $W$ is a smooth Osserman algebraic
curvature tensor, the characteristic polynomial of $W_{X|X^\perp}$ (of the restriction of the Jacobi operator
$W_X$ to the subspace $X^\perp$) does not depend on $X$ and is a well-defined smooth map
$p: M^n \to \br_{n-1}[t], \; y \to p_y(t)$, where $\br_{n-1}[t]$ is the $(n-1)$-dimensional affine space of
polynomials of degree $n-1$ with the leading term $(-t)^{n-1}$. As all the roots of $p_y(t)$ are real and the number
of different roots is constant on every connected component of $M'$, the eigenvalues $\mu_0, \mu_1, \dots, \mu_l$
of $W_{X|X^\perp}$ are smooth functions and their multiplicities $m_0, m_1, \dots, m_l$ are constant,
on every connected component of $M'$ (we chose the labelling in such a way that $m_0=\max(m_0,m_1,\dots, m_l)$.

First consider the case $n \ne 8$. The Weyl tensor has a Clifford structure given by \eqref{eq:cs} at every point
of $M'$. By Lemma~\ref{l:nvsnu}, for $n >4, \, n \ne 8,16, \; n-1-\nu > \nu$, for any Clifford structure on $\Rn$. By
\eqref{eq:radon}, for $n=16, \; \nu\le 8$, so by the assumption, the inequality $n-1-\nu > \nu$ also holds for $n=16$.
Then the biggest multiplicity of an eigenvalue of $W_{X|X^\perp}$ is $n-1-\nu$ (see Remark~\ref{rem:climpliesos}).
So the number $\nu=n-1-m_0$ is constant and the function $\la_0=\mu_0$ is smooth on every connected
component of $M'$. Moreover, for every smooth unit vector field $X$ on $M'$ and every $i=1, \dots, l$, the
$\mu_i$-eigendistribution of $W_{X|X^\perp}$ is $\Span_{j:\la_0+3\eta_j=\mu_i}(J_jX)$. As $\la_0$ and $\mu_i$ are
smooth functions on every connected component of $M'$, $\eta_j$ also is. Moreover, on every connected
component of $M'$, every distribution $\Span_{j:\la_0+3\eta_j=\mu_i}(J_jX)$ is smooth and has a constant
dimension $m_i$, for any nowhere vanishing smooth vector field $X$. By assertion~3 of
Lemma~\ref{l:level}, there exists a neighbourhood $\mU_i(x)$ and smooth anticommuting almost Hermitian
structures $\tilde J_j$ (for the $j$'s such that $\la_0+3\eta_j=\mu_i$) on $\mU_i(x)$ such that
$\Span_{j:\la_0+3\eta_j=\mu_i}(J_jX)=\Span_{j:\la_0+3\eta_j=\mu_i}(\tilde J_jX)$. Let $\tilde W$ be the algebraic
curvature tensor on $\mU= \cap_{i=1}^l\mU_i(x)$ with the Clifford structure
$\Cliff(\nu; \tilde J_1,\dots, \tilde J_\nu; \la_0, \eta_1, \dots , \eta_\nu)$. Then $\nu=n-1-m_0$ is constant and all
the $\tilde J_i, \eta_i$ and $\la_0$ are smooth on $\mU$. Moreover, for every unit vector field $X$ on
$\mU$, the Jacobi operators $\tilde W_X$ and $W_X$ have the same eigenvalues and eigenvectors by construction, hence
$\tilde W_X=W_X$, which implies $\tilde W=W$.

Now consider the case $n=8$. By Lemma~\ref{l:formofR}, at every point $x \in M'$, the Weyl tensor either has a
$\Cliff(3)$-structure, with $J_1J_2=J_3$, or a $\Cliff(7)$-structure (but not both). As on every connected component
$M_\a$ of $M'$, the number and the multiplicities of the eigenvalues of the operator $W_{X|X^\perp}$, $X \ne 0$, are
constant, it follows from Remark~\ref{rem:climpliesos} that the only case when $M_\a$ may potentially contain the
points of the both kinds is when one of the eigenvalues of $W_{X|X^\perp}, \; X \ne 0$, on $M_\a$ has multiplicity $4$
and the Clifford structure at every point $x \in M_\a$ is either $\Cliff(3;J_1, J_2, J_3;\la_0,\eta_1,\eta_2,\eta_3)$
with $J_1J_2=J_3$, or $\Cliff(7; J_1,\dots, J_7;\la_0-3\xi, \eta_1+\xi, \eta_2+\xi, \eta_3+\xi, \xi, \xi, \xi, \xi)$,
where $\eta_1, \eta_2, \eta_3 \ne 0$ (some of them can be equal) and $\xi \ne -\eta_i, 0$. The eigenvalues of
$W_{X|X^\perp}, \; \|X\|=1$, at every point $x \in M_\a$ are $\la_0$, of multiplicity $4$, and $\la_0+3\eta_i$.
Let $X$ be an arbitrary nowhere vanishing smooth vector field on a neighbourhood $\mU \subset M_\a$ of a point
$x \in M_\a$. Then the four-dimensional eigendistribution of the operator $W_{X|X^\perp}$ corresponding to the
eigenvalue of multiplicity $4$ is smooth, therefore its orthogonal complement, the distribution
$\Span(J_1X, J_2X, J_3X)$ is also smooth. By assertion~3 of Lemma~\ref{l:level}, there exist
smooth anticommuting almost Hermitian structures $\tilde J_1, \tilde J_2, \tilde J_3$ on (a possibly
smaller) neighbourhood $\mU$ such that $\Span(\tilde J_1X, \tilde J_2X, \tilde J_3X)=\Span(J_1X, J_2X, J_3X)$.
By assertion~1 of Lemma~\ref{l:level}, every $\tilde J_i$ is a linear combination of the $J_j$'s:
$\tilde J_i=\sum_{j=1}^3 a_{ij}J_j$, and moreover, the matrix $(a_{ij})$ must be orthogonal, as the $\tilde J_i$'s are
anticommuting almost Hermitian structures. It follows that $\tilde J_1\tilde J_2\tilde J_3=\pm J_1J_2J_3$. The 
operator on the left-hand side is smooth on $\mU$, the one on the right-hand side is $\pm \id_{\Ro}$, at the points 
where the Clifford structure is $\Cliff(3)$ with $J_1J_2=J_3$, and is symmetric with trace zero, at the points where
the Clifford structure is $\Cliff(7)$ (which follows from the identity $J_4(J_1J_2J_3)J_4=J_1J_2J_3$). Therefore
all the point of $\mU$ either have a $\Cliff(3)$-structure with $J_1J_2=J_3$, or a
$\Cliff(7)$-structure. In the both cases, the Clifford structure for $W$ can be taken smooth:
in the first case, we follow the arguments as in the first part of the proof, as $\nu < n-1-\nu$; in the second one,
we apply assertion~3 of Lemma~\ref{l:level} to every eigendistribution of $W_{X|X^\perp}$.

Thus for any $x \in M'$, the Weyl tensor on a neighbourhood $\mU=\mU(x)$ has the form \eqref{eq:cs}, with a constant
$\nu$ and smooth $\la_0, \eta_i$ and $J_i$. Then the curvature tensor has
the form \eqref{eq:confcs}, with the operator $\rho$ given by
$\rho=\frac{1}{n-2}\mathrm{Ric}+(\frac12 \la_0-\frac{\mathrm{scal}}{2(n-1)(n-2)})\id$,
where $\mathrm{Ric}$ is the Ricci operator and $\mathrm{scal}$ is the scalar curvature. As $\la_0$ is a smooth
function, the operator field $\rho$ is also smooth.
\end{proof}

\begin{remark}\label{rem:smoothcl}
In effect, the proof shows that if an algebraic curvature tensor $\mathcal{R}$ field has a Clifford structure at every
point of a Riemannian manifold, (and $\nu \ne 8$ when $n =16$) then it has a Clifford structure of the same class of
differentiability as $\mathcal{R}$ on a neighbourhood of every generic point of the manifold.
\end{remark}

\begin{remark} \label{rem:n8nu7}
As it follows from assertion~1 of Lemma~\ref{l:formofR} (in fact, from equation \eqref{eq:8sq}), in the case
$n=8, \; \nu=7$ we can replace in \eqref{eq:confcs} $\rho$ by $\rho- \frac32 f \, \id$ and $\eta_i$ by $\eta_i+f$,
without changing $R$, where $f$ is an arbitrary smooth function on $\mathcal{U}$ (if we want the resulting 
Clifford structure to be $\Cliff(7)$, we additionally require that $\eta_i+f$ is nowhere zero).
\end{remark}


Let $x\in M'$ and let $\mathcal{U}=\mathcal{U}(x)$ be the neighbourhood of $x$ defined in Lemma~\ref{l:locc1}.
By the second Bianchi identity, $(\n_UR)(X,Y) Y + (\n_YR)(U,X) Y+(\n_XR)(Y,U) Y = 0$. Substituting $R$ from
\eqref{eq:confcs} and using the fact that the operators $J_i$'s and their covariant derivatives are skew-symmetric and
the operator $\rho$ and its covariant derivatives are symmetric we get:
\begin{equation} \label{eq:confBZY}
\begin{split}
&\<X,Y\> ((\n_U \rho)Y - (\n_Y \rho)U)+\|Y\|^2 ((\n_X \rho)U-(\n_U \rho) X) +\<U,Y\> ((\n_Y \rho)X-(\n_X \rho) Y)\\
&+\<(\n_Y \rho)U-(\n_U \rho)Y, Y\> X +\<(\n_X \rho)Y-(\n_Y \rho)X, Y\> U +\<(\n_U \rho)X-(\n_X \rho)U, Y\> Y \\
&+ \sum\nolimits_{i=1}^\nu 3(X(\eta_i) \< J_iY, U \> - U(\eta_i) \< J_iY, X \>)J_iY  \\
&+ \sum\nolimits_{i=1}^\nu Y(\eta_i) (2 \< J_iU, X \> J_iY + \< J_iY, X \> J_iU - \< J_iY, U \> J_iX) \\
&+ \sum\nolimits_{i=1}^\nu \eta_i
\bigl((3 \< (\n_UJ_i)X, Y \> + 3 \< (\n_XJ_i)Y, U \> + 2 \< (\n_YJ_i)U, X \>)J_iY\\
&+3 \< J_iX, Y \> (\n_UJ_i) Y + 3 \< J_iY, U \> (\n_XJ_i) Y + 2 \< J_iU, X \> (\n_YJ_i) Y\\
&+ \< (\n_YJ_i) Y, X \> J_iU + \< J_i Y, X \> (\n_YJ_i) U
- \< (\n_YJ_i) Y, U \> J_iX - \< J_i Y, U \> (\n_YJ_i) X \bigr)=0.
\end{split}
\end{equation}
Taking the inner product of \eqref{eq:confBZY} with $X$ and assuming $X, Y$ and $U$ to be orthogonal we obtain
\begin{equation} \label{eq:confBZYX}
\begin{split}
&\|X\|^2 \<Q(Y),U\> +\|Y\|^2 \<Q(X),U\>\\
+& \sum\nolimits_{i=1}^\nu 3(X(\eta_i) \< J_iY, U \> - Y(\eta_i) \<J_iX,U\> - U(\eta_i) \<J_iY,X\>)\<J_iY,X\> \\
+& \sum\nolimits_{i=1}^\nu 3 \eta_i
\bigl((2 \< (\n_UJ_i)X, Y \> + \< (\n_XJ_i)Y, U \> + \< (\n_YJ_i)U, X \>)\<J_iY,X\>\\
& \hphantom{\sum\nolimits_{i=1}^\nu \la_i (}
 - \< J_iY, U\> \<(\n_XJ_i)X, Y\> - \< J_iX, U\> \<(\n_YJ_i) Y,X\>\bigr)=0,
\end{split}
\end{equation}
where $Q: \Rn \to \Rn$ is the quadratic map defined by
\begin{equation}\label{eq:defQ}
\<Q(X), U\>=\<(\n_X \rho)U-(\n_U \rho) X,X\>.
\end{equation}
Note that $\<Q(X), X\> = 0$.

{
\begin{lemma}\label{l:nablarho}
In the assumptions of Lemma~\ref{l:locc1}, let $x \in M'$ and let $\mU$ be the corresponding neigbourhood of $x$.
Suppose that if $n = 8$, then $\nu=3$ and $J_1J_2=J_3$ on $\mU$, and if $n=16$, then $\nu \le 4$.
For every point $y \in \mU$, identify $T_yM^n$ with the Euclidean space $\Rn$ via a linear isometry.
Then
\begin{enumerate}[\rm (i)]
  \item
there exist vectors $m_i, b_{ij} \in \Rn, \; i,j =1, \dots, \nu$, such that for all $X, Y, U \in \Rn$, and all
$i =1, \dots, \nu$,
\begin{subequations}\label{eq:lnablarhoi}
\begin{gather} \label{eq:Qsum}
Q(Y)=3 \sum\nolimits_{k=1}^\nu \<m_k,Y\>J_kY, \\
\label{eq:nXJX}
    (\n_X J_i)X = \eta_i^{-1}(\|X\|^2 m_i-\<m_i,X\>X)+\sum\nolimits_{j=1}^\nu \<b_{ij},X\>J_jX,\\
\label{eq:bijbji}
    b_{ij}+b_{ji}=\eta_i^{-1}J_jm_i+\eta_j^{-1}J_im_j, \\
\label{eq:neta}
    \nabla \eta_i = 2 J_im_i,\\
\label{eq:JiJkY}
\sum\nolimits_{j \ne i}(\<\eta_i b_{ij} + \eta_j b_{ji}, J_iY\> J_jY+\<\eta_i b_{ij} + \eta_j b_{ji},Y\> J_iJ_jY)=0.
\end{gather}
\end{subequations}

  \item
  the following equations hold:
\begin{subequations}\label{eq:lnablarhoii}
\begin{gather}
\label{eq:skewrho}
(\n_Y \rho)U-(\n_U \rho)Y=\sum\nolimits_{i=1}^\nu (2\<J_iY,U\>m_i - \<m_i,Y\>J_iU+\<m_i,U\>J_iY),\\
\label{eq:bijne}
b_{ij} (3-\eta_i \eta_j^{-1})+b_{ji} (3-\eta_j \eta_i^{-1})=0, \quad i \ne j,\\
\label{eq:allequal}
J_i m_i = \eta_i p, \quad i=1, \ldots, \nu, \quad \text{for some $p \in \Rn$}.
\end{gather}
\end{subequations}
\end{enumerate}
\end{lemma}
\begin{proof}
(i) We split the proof of this assertions into the two cases:
the \emph{exceptional case}, when either $n=6,\, \nu=1$, or $n=12, \, \nu=3, \; J_1J_2=\pm J_3$, or
$n = 8, \; \nu=3, \; J_1J_2=J_3$, and the \emph{generic case}: all the other Clifford structures considered in
the lemma.

\smallskip

\underline{Generic case.}
From \eqref{eq:confBZYX} we obtain
\begin{equation} \label{eq:nrho}
\|X\|^{-2} \<Q(X), U\> +\|Y\|^{-2} \<Q(Y), U\>=0, \qquad \text{for all } X \perp \cI Y, \;
X, Y \perp \cI U, \; X, Y, U \ne 0.
\end{equation}
We want to show that $\<Q(X), U\> = 0$, for all $X \perp \cI U$. This is immediate when $n > 3\nu+3$. Indeed, for
any $U \ne 0$ and any unit $X\perp \cI U, \quad \mathrm{codim}(\cI U + \cI X) > \nu +1$,
so we can choose unit vectors $Y_1, Y_2 \perp \cI U + \cI X$ such that $Y_1 \perp \cI Y_2$. Then \eqref{eq:nrho}
implies that $\<Q(X), U\>=-\<Q(Y_1), U\>=\<Q(Y_2), U\>=-\<Q(X), U\>$.

Consider the case $n \le 3\nu+3$. By assertion~(i) of Lemma~\ref{l:nvsnu}, this could only happen when
$n=12$, $\nu=3$ or $n=24, \, \nu = 7$ (for the pairs $(n,\nu)$ belonging to the generic case), and in the both cases
$n = 3\nu+3$.
Choose and fix an arbitrary $U \ne 0$ and consider the quadratic form $q(X) = \<Q(X), U\>$ defined
on the $(2\nu+2)$-dimensional space $L=(\cI U)^\perp$. Assume $q \ne 0$. By \eqref{eq:nrho}, the restriction of $q$ to
the unit sphere of $L$ is not a constant, so it attains its maximum (respectively minimum) on a great sphere $S_1$
(respectively $S_2$). The subspaces $L_1$ and $L_2$ defined by $S_1$ and $S_2$ are orthogonal. Moreover, by
\eqref{eq:nrho}, $L_2 \supset (\cI X)^\perp \cap L$, for any nonzero $X \in L_1$, which implies that
$\dim L_2 \ge \nu+1$. Similarly $\dim L_1 \ge \nu+1$, so, as $L_1 \perp L_2, \quad \dim L_1=\dim L_2 = \nu+1$, and
$L=L_1 \oplus L_2$. It follows that for some $c > 0, \quad q(X)=c(\|\pi_1X\|^2-\|\pi_2X\|^2)$, where
$\pi_i:L\to L_i$ is the orthogonal projection. Moreover,
$L_2=(\cI X)^\perp \cap L$, for all nonzero $X \in L_1$, which means that the subspace $L_1=L_2^\perp \cap L$
(and similarly $L_2$) is $\pi  \cI$-invariant, where $\pi: \Rn \to L$ is the orthogonal projection, and even more:
$\pi  \cI X = L_\a$, for every nonzero $X \in L_\a, \; \a=1,2$, by the dimension count. Let
$X=X_1+X_2,\, Y=Y_1+Y_2 \in L$, where $X_\a=\pi_\a X,\; Y_\a=\pi_\a Y$. The condition $Y \perp \cI X$ is equivalent to
$\<X_1, Y_1\> + \<X_2, Y_2\>= \<\pi  J_i X_1, Y_1\> + \<\pi  J_i X_2, Y_2\>=0$, for all $i=1, \ldots \nu$.
Take arbitrary orthonormal bases for $L_1$ and for $L_2$ and denote $M_\a(X_\a), \; \a=1,2$, the
$(\nu+1)\times(\nu+1)$-matrix whose columns relative to the chosen basis for $L_\a$ are
$X_\a, \pi  J_1 X_\a, \dots, \pi  J_\nu X_\a$. Then $Y \perp \cI X$ if and only if
$M_1(X_1)^tY_1=-M_2(X_2)^tY_2$. Since for $\a=1,2$, and any nonzero $X_\a \in L_\a$, the columns of $M_\a(X_\a)$ span
$L_\a$, we obtain $Y_2= -(M_2(X_2)^t)^{-1}M_1(X_1)^tY_1$, for any $X_2 \ne 0$. Then, as
$q(X)=c(\|X_1\|^2-\|X_2\|^2), \; q(Y)=c(\|Y_1\|^2-\|Y_2\|^2)$, equation \eqref{eq:nrho} implies
$\|Y_1\|^2\|X_1\|^2-\|Y_2\|^2\|X_2\|^2=0$, so $\|Y_1\|^2\|X_1\|^2-\|(M_2(X_2)^t)^{-1}M_1(X_1)^tY_1\|^2\|X_2\|^2=0$,
for any $X_1, Y_1 \in L_1$ and any nonzero $X_2 \in L_2$. It follows that
$\|X_1\|^2 (M_1(X_1)^tM_1(X_1))^{-1}=\|X_2\|^2 (M_2(X_2)^tM_2(X_2))^{-1}$, for any nonzero $X_\a \in L_\a$. Thus for
some positive definite symmetric 
$(\nu+1)\times(\nu+1)$-matrix $T$, we have $M_\a(X_\a)^tM_\a(X_\a)=\|X_\a\|^2 T$, for all
$X_\a \in L_\a, \; \a=1,2$. Then for any $X = X_1+X_2 \in L,\; X_\a \in L_\a$, and any $i=1, \dots, \nu$, 
$\|\pi  J_i X\|^2= \|\pi  J_i X_1\|^2+ \|\pi  J_i X_2\|^2=(M_1(X_1)^tM_1(X_1)+M_2(X_2)^tM_2(X_2))_{ii}=
T_{ii}(\|X_1\|^2+ \|X_2\|^2)=T_{ii}\|X\|^2$. On the other hand, for any
$X \in L, \; \pi  J_i X=J_i X-\|U\|^{-2}\sum_{j=1}^\nu\<J_iX,J_jU\>J_jU$, so
$\|\pi  J_i X\|^2= \|X\|^2-\|U\|^{-2}\sum_{j=1}^\nu\<J_iX,J_jU\>^2$. It follows that
$\|X\|^2\|U\|^2(1-T_{ii})= \sum_{j=1}^\nu\<J_iX,J_jU\>^2=\sum_{j=1}^\nu\<X,J_iJ_jU\>^2$, for an arbitrary $X \in L$.
As $\dim L = 2\nu+2 > \nu$, we can choose a nonzero $X \in L$ orthogonal to the $\nu$ vectors
$J_iJ_jU, \; j=1, \ldots, \nu$. This implies $T_{ii}=1$, so $X \perp J_iJ_jU$, for all $i, j = 1, \ldots, \nu$ and all
$X \in L = (\cI U)^\perp$. Therefore $J_iJ_jU \in \cI U$, for all $i, j = 1, \ldots, \nu$ and all $U \in \Rn$
for which the quadratic form $q(X) = \<Q(X), U\>$ defined on $(\cI U)^\perp$ is nonzero. If this is true for at least
one $U$, then this is true for a dense subset of $\Rn$, which implies that $J_iJ_jU \in \cI U$, for all
$i, j = 1, \ldots, \nu$ and all $U \in \Rn$. Then by assertion~1 of Lemma~\ref{l:level}, for
$i \ne j, \; J_iJ_jU=\sum_{k=1}^\nu a_{ijk} J_kU$, for some constants $a_{ijk}$, which implies that
$\<J_kJ_iJ_jU, U\> = a_{ijk}\|U\|^2$, so for all the triples of pairwise distinct $i,j,k$, the symmetric operator
$J_kJ_iJ_j$ on $\Rn$ is a multiple of the identity. This is impossible when $\nu > 3$ (as for $l \ne i,j,k$, the
operator $J_lJ_kJ_iJ_j$ must be orthogonal and symmetric). The only remaining cases are $n=12, \; \nu=3$, with
$J_1J_2J_3 = \pm \id$, and $n=6, \nu =1$, which are considered under the exceptional case below.

Thus  $\<Q(X),U\>=0$, for $X \perp \cI U$, so $Q(X) \in \cI X$, for all $X \in \Rn$. By assertion~1 of
Lemma~\ref{l:level} (and the fact that $\<Q(X),X\>=0$), this implies equation \eqref{eq:Qsum},
with some vectors $m_i \in \Rn$.

To prove \eqref{eq:nXJX} and \eqref{eq:bijbji}, we first show that for an arbitrary $X \ne 0$, there is a dense
subset of the $Y$'s in $(\cI X)^\perp$ such that $\cJ X \cap \cJ Y = 0$. This follows from the dimension count
(compare to \cite[Lemma~3.2 (1)]{Nhjm}). For $X \ne 0$, define the cone
$\mathcal{C}X=\{J_uJ_vX \, : \, u, v \in \br^\nu\}$ (see \eqref{eq:ortmult}). As
$\dim \mathcal{C}X \le 2\nu-1 < n-(\nu+1)= \dim (\cI X)^\perp$ (the inequality in the middle follows from
assertion~(i) of Lemma~\ref{l:nvsnu}), the complement to $\mathcal{C}X$ is dense in $(\cI X)^\perp$. This complement
is the required subset, as the condition $Y \notin \mathcal{C}X$ is equivalent to $\cJ X \cap \cJ Y = 0$. Substituting
such $X, Y$ into \eqref{eq:confBZYX} we obtain by \eqref{eq:Qsum}:
\begin{equation*}
\sum\nolimits_{i=1}^\nu (\|X\|^2\<m_i,Y\>-\eta_i \<(\n_XJ_i)X, Y\>) J_iY
+\sum\nolimits_{i=1}^\nu (\|Y\|^2\<m_i,X\>-\eta_i \<(\n_YJ_i)Y, X\>) J_iX =0.
\end{equation*}
As $\cJ X \cap \cJ Y = 0$, all the coefficients vanish, so
$\|X\|^2\<m_i,Y\>-\eta_i \<(\n_XJ_i)X, Y\>=0$, for all $X \in \Rn$, all $i=1,\dots, \nu$, and all $Y$ from a
dense subset of $(\cI X)^\perp$, which implies that $(\n_XJ_i)X -\eta_i^{-1}\|X\|^2 m_i \in \cI X$, for all
$X \in \Rn$. Equation \eqref{eq:nXJX} then follows from
assertion~1 of Lemma~\ref{l:level}. Equation \eqref{eq:bijbji} follows from \eqref{eq:nXJX} and the fact that
$\<(\n_X J_i)X,J_jX\>+ \<(\n_X J_j)X,J_iX\>=0$.

To prove \eqref{eq:neta} and \eqref{eq:JiJkY}, substitute $X=J_kY, \; U \perp X, Y$ into \eqref{eq:confBZYX}.
Consider the first term in the second summation. As $\<J_iY, X\> = \|Y\|^2 \K_{ik}$, that term equals
$3 \eta_k(2 \< (\n_UJ_k)X, Y \> + \< (\n_XJ_k)Y, U \> + \< (\n_YJ_k)U, X \>)\|Y\|^2$. As $J_k$
is orthogonal and skew-symmetric,
$\< (\n_UJ_k)X, Y \>=\< (\n_UJ_k)J_kY, Y \>=-\< J_k(\n_UJ_k)Y, Y \>=\<(\n_UJ_k)Y, J_kY \>=0$.
Next, $\<(\n_YJ_k)U, X \>=-\<(\n_YJ_k)J_kY, U\>=\<J_k(\n_YJ_k)Y, U\>$ $=
\<(\eta_k^{-1}\|Y\|^2 J_km_k+\sum\nolimits_{j=1}^\nu \<b_{kj},Y\>J_kJ_jY, U\>$ by \eqref{eq:nXJX}. Similarly, as
$Y=-J_kX$, it follows from \eqref{eq:nXJX} that $\< (\n_XJ_k)Y, U \>=\< J_k(\n_XJ_k)X, U \>=
\<J_k(\eta_k^{-1}(\|X\|^2 m_k-\<m_k,X\>X)+\sum\nolimits_{j=1}^\nu \<b_{kj},X\>J_jX),U\>=
\<\eta_k^{-1}\|Y\|^2 J_km_k+\sum\nolimits_{j \ne k} \<b_{kj},J_kY\>J_jY-\<b_{kk},J_kY\>J_kY,U\>$.
Substituting this into \eqref{eq:confBZYX} and using \eqref{eq:Qsum} and \eqref{eq:nXJX}
we obtain after simplification:
\begin{equation} \label{eq:XJkY}
\|Y\|^2 \<2J_km_k-U(\eta_k))+
 \sum\nolimits_{j=1}^\nu \<\eta_k b_{kj} + \eta_j b_{jk}, \<J_jY,U\>J_kY +\<J_kJ_jY,U\>Y\> =0.
\end{equation}
By \cite[Lemma~3.2(3)]{Nhjm}, for all $U \in \Rn$, we can find a nonzero $Y$ such that $U \perp \cJ Y + \cJ J_kY$.
Substituting such a $Y$ into \eqref{eq:XJkY} proves \eqref{eq:neta}. Then \eqref{eq:XJkY} simplifies to
\eqref{eq:JiJkY}.

\smallskip

\underline{Exceptional case} (either $n=6,\, \nu=1$, or $n=12, \, \nu=3, \; J_1J_2=\pm J_3$, or
$n = 8, \; \nu=3, \; J_1J_2=J_3$).

In all these cases, the Clifford structure has the following ``$J^2$-property": for every
$X \in \Rn, \; \cI \cI X = \cJ \cI X = \cI X$. In particular, if $Y \perp \cI X$, then $\cI Y \perp \cI X$.

Substitute $X=J_kU$ and $Y \perp \cI X= \cI U$ to \eqref{eq:confBZY} and take the inner product of the resulting
equation with $J_kY$. Using the fact that $\<(\n_YJ_k)U,J_kU \>=\<(\n_YJ_k)Y,J_kY \>=0$ and the $J^2$-property we get
\begin{equation*}
-J_k((\n_{J_kU} \rho)U-(\n_U \rho) J_kU) + 2\|U\|^2 \n \eta_k + 3\eta_k ( (\n_UJ_k)J_kU - (\n_{J_kU}J_k)U) \in \cI U.
\end{equation*}
The expression $F(U)$ on the left-hand side is a quadratic map from $\Rn$ to itself. By assertion~1 of
Lemma~\ref{l:level}, $F(U)$ is a linear combination of $U, J_1U, \dots, J_\nu U$ whose coefficients are
linear forms of $U$. In particular, the cubic polynomial $\<F(U),J_kU\>$ must be divisible by $\|U\|^2$.
As $J_k$ is orthogonal and skew-symmetric, $\<(\n_UJ_k)J_kU - (\n_{J_kU}J_k)U, J_kU\> =0$, so there exists a vector
$m_k \in \Rn$ such that $\<(\n_{J_kU} \rho)U-(\n_U \rho) J_kU, U\>=-3\|U\|^2\<m_k,U\>$. It follows
that the quadratic map $Q$ defined by \eqref{eq:defQ} satisfies $\<Q(U),J_kU\>=3\|U\|^2\<m_k,U\>$,
for all $U \in \Ro$ and all $k=1,\dots,\nu$. As $\<Q(U), U\>=0$, we can define a quadratic map $T: \Rn \to \Rn$
such that for all $U \in \Rn$,
\begin{equation}\label{eq:QT3}
Q(U)=T(U)+3\sum\nolimits_{k=1}^\nu \<m_k,U\> J_kU, \qquad T(U) \perp \cI U.
\end{equation}
Taking $U = J_kX, \; X,U \perp \cI Y$ in \eqref{eq:confBZYX} and using \eqref{eq:QT3} we obtain
$-J_kT(Y) +3\|Y\|^2 m_k- 3\eta_k (\n_YJ_k) Y \in \cI Y$. From assertion~1 of Lemma~\ref{l:level} it follows
that the expression on the left-hand side is a linear combination of $Y, J_1Y, \dots , J_\nu Y$ whose coefficients are
linear forms of $Y$, so for some vectors $b_{ij} \in \Rn$,
\begin{equation}\label{eq:nYY83withT}
    (\n_YJ_i) Y=\eta_i^{-1}(m_i \|Y\|^2-\<m_i,Y\>Y)-(3\eta_i)^{-1}J_iT(Y)+\sum\nolimits_{j=1}^\nu \<b_{ij},Y\> J_jY.
\end{equation}
As $\<(\n_YJ_i) Y, J_jY\>$ is antisymmetric in $i$ and $j$ and $J_iT(Y) \perp \cI Y$ by \eqref{eq:QT3}
and the $J^2$-property, the $b_{ij}$'s satisfy \eqref{eq:bijbji}.

Take $X=J_kY, \; U \perp \cI Y = \cI X$ in \eqref{eq:confBZYX}. As $\< (\n_UJ_k)J_kY, Y \>=0$,
$\< (\n_YJ_k)U, X \>=-\< (\n_YJ_k)J_kY, U \>$ $=\< J_k(\n_YJ_k)Y, U \>$, and similarly,
$\< (\n_XJ_k)Y, U \>=-\< (\n_XJ_k)J_kX, U \>$ $=\< J_k(\n_XJ_k)X, U \>$, we obtain from
(\ref{eq:QT3}, \ref{eq:nYY83withT}) after simplification that
\begin{equation} \label{eq:confBZYX83}
2T(Y) + 2T(J_kY) -3\|Y\|^2 (\n \eta_k -2J_km_k) \in \cI Y.
\end{equation}
In the case $n=6, \; \nu=1$, we can prove the remaining identities (\ref{eq:Qsum}, \ref{eq:nXJX}, \ref{eq:neta},
\ref{eq:JiJkY}) of assertion~(i) as follows. Taking in \eqref{eq:confBZYX}
nonzero $X, Y, U$ such that the subspaces $\cI X, \cI Y$ and $\cI U$ are mutually orthogonal we obtain by
\eqref{eq:QT3} $\|X\|^{-2} \<T(X), U\> +\|Y\|^{-2} \<T(Y), U\>=0$ (which is, essentially, \eqref{eq:nrho}).
Replacing $Y$ by $J_1Y$ and using \eqref{eq:confBZYX83} we get $2T(X)+3\|X\|^2(\n \eta_1 -2J_1m_1) \in \cI X$.
The same is true with $X$ replaced by $J_1X$. Then by \eqref{eq:confBZYX83}, $\n \eta_1 -2J_1m_1 \in \cI X$, for all
$X \in \br^6$, so $\n \eta_1 -2J_1m_1=0$ (which is \eqref{eq:neta}). Then $T(X) \in \cI X$, hence $T(X)=0$, as
$T(X) \perp \cI X$ by \eqref{eq:QT3}. Now \eqref{eq:Qsum} follows from \eqref{eq:QT3},
\eqref{eq:nXJX} follows from \eqref{eq:nYY83withT}, and \eqref{eq:JiJkY} is trivially satisfied, as $\nu=1$.

In the cases $n=8, 12, \; \nu=3, \; J_1J_2=J_3$ (if $J_1J_2=-J_3$, we replace $J_3$ by $-J_3$, without changing
the curvature tensor \eqref{eq:confcs}), we argue as follows.
Adding equation \eqref{eq:confBZYX83} with $k=1$ and with $k=2$ and then subtracting \eqref{eq:confBZYX83} with $k=3$
and $Y$ replaced by $J_1Y$ we get
$4T(Y) -3\|Y\|^2 ((\n \eta_1 -2J_1m_1)$ $+ (\n \eta_2 -2J_2m_2)-(\n \eta_3 -2J_3m_3)) \in \cI Y$.
This remains true under a cyclic permutation of the indices $1,2,3$, which implies
$(\n \eta_k -2J_km_k)-(\n \eta_i -2J_im_i) \in \cI Y$, for all $i,k = 1,2,3$ and all $Y \in \Rn$. Then
$\n \eta_k -2J_km_k=\n \eta_i -2J_im_i = \tfrac43 V$, for some vector $V \in \Rn$, and $T(Y) -\|Y\|^2 V \in \cI Y$
from the above. As $T(Y) \perp \cI Y$ by \eqref{eq:QT3},
we obtain $T(Y)=\|Y\|^2V - \<Y,V\>Y-\sum\nolimits_{i=1}^3 \<J_iY,V\>J_iY$, so
\begin{equation}\label{eq:83withV}
\begin{gathered}
    \n \eta_i = 2J_im_i + \tfrac43 V, \qquad Q(Y)=\|Y\|^2V - \<Y,V\>Y+\sum\nolimits_{j=1}^3 \<3m_j+J_jV,Y\>J_jY, \\
    (\n_YJ_i) Y=(3\eta_i)^{-1} \bigl(\|Y\|^2(3m_i -J_iV) - \<3m_i-J_iV,Y\>Y
    +\sum\nolimits_{j=1}^3 \<3\eta_i b_{ij}- J_jJ_iV,Y\> J_jY \bigr)
\end{gathered}
\end{equation}
(the second equation follows from \eqref{eq:QT3}; the third one, from \eqref{eq:nYY83withT} and the fact that
$J_1J_2=J_3$).

Substitute $X=J_kY$ into \eqref{eq:confBZYX} again, with an arbitrary $U \perp X, Y$. Using \eqref{eq:83withV} and
the fact that the $J_i$'s are skew-symmetric, orthogonal and anticommute, we obtain after simplification:
\begin{equation*}
\sum\nolimits_{i=1}^3 \<3a_{ik}- 2J_iJ_kV,J_kY\> J_iY
+ \sum\nolimits_{i=1}^3 \<3a_{ik}- 2J_iJ_kV,Y\> J_kJ_iY \in \Span(Y, J_kY),
\end{equation*}
where $a_{ik}=\eta_k b_{ki}+\eta_i b_{ik}$. Taking $k=1$ and using the fact that $J_1J_2=J_3$ we get from the
coefficient of $J_2Y: \; 3J_1a_{12}-4J_2V+3a_{13}=0$, so $4V=-3J_2a_{13}+3J_3a_{12}$. Cyclically permuting the indices
$1,2,3$ and using the fact that $a_{ik}=a_{ki}$ we obtain $V=0$, which implies \eqref{eq:JiJkY}. As $V=0$, equations
(\ref{eq:Qsum}, \ref{eq:neta}, \ref{eq:nXJX}) follow from \eqref{eq:83withV}.

\smallskip

(ii) By \eqref{eq:defQ} and \eqref{eq:Qsum},
$\<(\n_X \rho)U-(\n_U \rho) X,X\>=3\sum_{i=1}^\nu \<m_i,X\>\<J_iX,U\>$, for all $X, U \in \Rn$. Polarizing this
equation and using the fact that the covariant derivative of $\rho$ is symmetric we obtain
$\<(\n_X\rho)U,Y\>+\<(\n_Y \rho)U,X\>-2\<(\n_U\rho) Y,X\>=3\sum_{i=1}^\nu (\<m_i,Y\>\<J_iX,U\>+\<m_i,X\>\<J_iY,U\>)$.
Subtracting the same equation, with $Y$ and $U$ interchanged, we get $\<(\n_Y \rho)U-(\n_U \rho)Y,X\>=
\sum_{i=1}^\nu (2\<m_i,X\>\<J_iY,U\>$ $+\<m_i,Y\>\<J_iX,U\>- \<m_i,U\>\<J_iX,Y\>)$, which proves \eqref{eq:skewrho}.

To establish \eqref{eq:bijne}, substitute $X \perp \cI Y, \; U = J_kY$ into \eqref{eq:confBZY}. Using the equations of
assertion~(i) and \eqref{eq:skewrho} we obtain after simplification:
\begin{equation*}
3 (\n_XJ_k) Y - (\n_YJ_k) X =  -3\eta_k^{-1}\<m_k,Y\>X +
\sum\nolimits_{i=1}^\nu \eta_k^{-1} \<\eta_i b_{ik} + 2 \K_{ik} J_km_k, Y\> J_iX \mod( \cI Y ).
\end{equation*}
Subtracting three times polarized equation \eqref{eq:nXJX} (with $i=k$) and solving for $(\n_YJ_k) X$ we get
\begin{equation}\label{eq:nyjx}
(\n_YJ_k) X= \sum\nolimits_{i=1}^\nu \tfrac14 \eta_k^{-1}\<3 \eta_k b_{ki}-\eta_i b_{ik}- 2 \K_{ik} J_km_k,Y\>J_iX
\mod( \cI Y ),
\end{equation}
for all $X \perp \cI Y$. Choose $s \ne k$ and define the subset $S_{ks} \subset \Rn \oplus \Rn$ by
$S_{ks}=\{(X, Y) \, : \, X, Y \ne 0$, $X, J_kX, J_sX \perp \cJ Y \}$. It is easy to see that
$(X,Y) \in S_{ks} \Leftrightarrow (Y,X) \in S_{ks}$ and that replacing $\cJ Y$ by $\cI Y$ in the definition of
$S_{ks}$ gives the same set $S_{ks}$. Moreover, the set $\{X \, : \, (X,Y) \in S_{ks}\}$ (and hence the set
$\{Y \, : \, (X,Y) \in S_{ks}\}$) spans $\Rn$. If $n=8, \; \nu=3, \; J_1J_2=J_3$, this easily follows from the
$J^2$-property; in all the other cases, from \cite[Lemma~3.2 (4)]{Nhjm}. For $(X, Y) \in S_{ks}$, take the inner
product of \eqref{eq:nyjx} with $J_sX$. Since $\<(\n_YJ_k) X, J_sX\>$ is antisymmetric in $k$ and $s$, we get
$\<(3-\eta_k\eta_s^{-1})b_{ks}+(3-\eta_s\eta_k^{-1})b_{sk}, Y\>=0$, for a set of the $Y$'s spanning $\Rn$. This
proves \eqref{eq:bijne}.

To prove \eqref{eq:allequal}, we apply assertion~2 of Lemma~\ref{l:level}
to equation \eqref{eq:JiJkY}. If $\nu=1$, there is nothing to prove (in fact, if $\nu=1$ and $n \ge 8$,
the claim of Theorem~\ref{t:cocl} follows from \cite[Theorem~1.1]{BG1}).
If $\eta_i b_{ij} + \eta_j b_{ji}=0$ for all $i \ne j$, then by \eqref{eq:bijne},
$b_{ij}+b_{ji}=0$ for all $i \ne j$, so by \eqref{eq:bijbji}, $\eta_i^{-1}J_jm_i=-\eta_j^{-1}J_im_j$.
Acting by $J_iJ_j$ we obtain that the vector $\eta_i^{-1}J_im_i$ is the same, for all $i=1, \dots, \nu$.

The only remaining possibility is $\nu =3$, $J_1J_2=J_3$ (if $J_1J_2=-J_3$ we can replace $J_3$ by $-J_3$ without
changing the curvature tensor \eqref{eq:confcs}), and $\eta_k b_{ki} + \eta_i b_{ik} = J_j v$, for all the triples
$\{i,j,k\}=\{1,2,3\}$, where $v \ne 0$. We will show that this leads to a contradiction.
Note that by \eqref{eq:radon}, the existence of a $\Cliff(3)$-structure implies that $n$ is divisible by $4$, so
by the assumption of the lemma, $n \ge 8$.

If $\eta_i=\eta_k$ for some $i \ne k$, then from \eqref{eq:bijne} and $\eta_k b_{ki} + \eta_i b_{ik} = J_j v$ it
follows that $v =0$, a contradiction. Otherwise, if the $\eta_i$'s are pairwise distinct, we get
$b_{ik}= (3 \eta_i - \eta_k)(4 \eta_i (\eta_i-\eta_k))^{-1} J_jv$ for $\{i,j,k\}=\{1,2,3\}$. Substituting this
to \eqref{eq:bijbji} and acting by $J_j$ on the both sides we get
$\eta_i^{-1} J_im_i-\eta_k^{-1} J_km_k=\frac14 \varepsilon_{ik}(\eta_i^{-1} +\eta_k^{-1}) v$, for
$\{i,j,k\}=\{1,2,3\}$, where for $i \ne k$ we define $\varepsilon_{ik} = \pm 1$ by $J_iJ_k=\varepsilon_{ik}J_j$.
It is easy to see that $\ve_{jk}=-\ve_{jk}$ and $\ve_{jk}=\ve_{ij}$, where $\{i,j,k\}=\{1,2,3\}$. Then
$\sum_{i=1}^3 \eta_i^{-1}=0$ and
$\eta_i^{-1} J_im_i=\frac{1}{12}\ve_{jk}(\eta_j^{-1}-\eta_k^{-1}) v+w$, for some $w \in \Rn$. It then follows
from \eqref{eq:neta} that $\n \eta_i=\frac{1}{6}\ve_{jk}\eta_i(\eta_j^{-1}-\eta_k^{-1}) v+2\eta_i w$, which implies
$\n \ln|\eta_1\eta_2\eta_3|=6w$ and $\n \ln|\eta_i\eta_j^{-1}|=-\frac12 \ve_{ij} \eta_k^{-1}v$. Let
$\mathcal{U}' \subset \mathcal{U}$ be a neighbourhood of $x$
on which $\n \ln|\eta_1\eta_2^{-1}| \ne 0$. Then $v$ is a nowhere vanishing smooth vector field on $\mathcal{U}'$.
Multiplying the metric on $\mathcal{U}$ by a function $e^f$ changes neither the Weil tensor, nor the $J_i$'s, and
multiplies every $\eta_i$ by $e^{-f}$ and $\n$ acting on functions by $e^{-f}$. Taking
$f=\frac13 \ln|\eta_1\eta_2\eta_3|$ we can assume that $w=0$ on
$\mathcal{U}'$, so that $C=\eta_1\eta_2\eta_3$ is a constant. Then, as $\sum_{i=1}^3 \eta_i^{-1}=0$, we get
$\n \eta_i=\pm\frac{1}{6}\sqrt{1-4C^{-1}\eta_i^3} v$. It follows that $v =\n t$ for some smooth function
$t:\mathcal{U}' \to \br$ such that $\eta_i= -36 C \wp(t+c_i)$, where $\wp$ is the Weierstrass function satisfying
$(\frac{d}{dt}\wp(t))^2=4 \wp(t)^3+6^{-6}C^{-2}$ and $c_i \in \br$. Summarizing the identities of this paragraph,
we have pointwise pairwise nonequal functions $\eta_i:\mathcal{U}' \to \br \setminus\{0\}$ satisfying
\begin{equation}\label{eq:summarize}
\begin{gathered}
    v=\n t \ne 0, \quad \n \eta_i=\tfrac{1}{6}\ve_{jk}\eta_i(\eta_j^{-1}-\eta_k^{-1}) v, \quad
    \sum\nolimits_{i=1}^3 \eta_i^{-1}=0, \quad \prod\nolimits_{i=1}^3 \eta_i=C=\mathrm{const}, \\
    m_i=-\tfrac{1}{12}\ve_{jk}\eta_i(\eta_j^{-1}-\eta_k^{-1}) J_iv, \quad
    b_{ii}= \tfrac{1}{12}\ve_{jk}(\eta_j^{-1}-\eta_k^{-1}) v, \quad
    b_{ij}= (3 \eta_i - \eta_j)(4 \eta_i (\eta_i-\eta_j))^{-1} J_kv,
\end{gathered}
\end{equation}
for $\{i,j,k\}=\{1,2,3\}$, where we used \eqref{eq:bijbji} to compute $b_{ii}$. Then equation \eqref{eq:nyjx}
simplifies to
$(\n_YJ_k) X= \sum_{i \ne k} \tfrac12 (\eta_k-\eta_i)^{-1} \<J_jv,Y\>J_iX \mod( \cI Y )$, for all $X \perp \cI Y$.
By the $J^2$-property, $\cI Y \perp \cI X$, so to find the ``$\mathrm{mod}( \cI Y )$"-part, we have to compute
the inner products of $(\n_YJ_k) X$ with $Y, J_1Y, J_2Y, J_3Y$. Since
$\<(\n_YJ_k) X,Y\>=-\<(\n_YJ_k) Y,X\>$, $\<(\n_YJ_k) X,J_kY\>=-\<(\n_YJ_k) J_kY,X\>=\< J_k(\n_YJ_k)Y,X\>$, and
$\<(\n_YJ_k) X,J_iY\>=-\<(\n_YJ_k) J_iY,X\>=-\< (\ve_{ki}(\n_YJ_j)-J_k(\n_YJ_i))Y,X\>$ (from $J_kJ_i=\ve_{ki}J_j$),
these products can be found using \eqref{eq:nXJX}. Simplifying by \eqref{eq:summarize} we get
\begin{multline*}
(\n_YJ_k) X= \tfrac{1}{12} \ve_{ij}(\eta_i^{-1}-\eta_j^{-1})(\<J_kv,X\>Y+\<v,X\>J_kY)\\
+\tfrac14 \eta_k^{-1}\sum\nolimits_{i \ne k}\<J_jv,X\>J_iY   
+\sum\nolimits_{i \ne k} \tfrac12 (\eta_k-\eta_i)^{-1} \<J_jv,Y\>J_iX,
\end{multline*}
for all $X \perp \cI Y$, where $\{i,j,k\}=\{1,2,3\}$. To compute $(\n_YJ_k) X$ when $X \in \cI Y$ we again use
\eqref{eq:nXJX} and the fact that $(\n_YJ_k)J_k=-J_k(\n_YJ_k)$ and $(\n_YJ_k) J_i=\ve_{ki}(\n_YJ_j)-J_k(\n_YJ_i)$,
for $\{i,j,k\}=\{1,2,3\}$. Simplifying by \eqref{eq:summarize} and using the above equation we get after some
calculations:
\begin{multline*}
(\n_YJ_k) X= \tfrac{1}{12} \ve_{ij}(\eta_i^{-1}-\eta_j^{-1})(\<J_kv,X\>Y+\<v,X\>J_kY-\<X,Y\>J_kv-\<X,J_kY\>v)\\
+\tfrac14 \eta_k^{-1}\sum\nolimits_{i \ne k}(\<J_jv,X\>J_iY-\<J_iY,X\>J_jv)
+\sum\nolimits_{i \ne k} \tfrac12 (\eta_k-\eta_i)^{-1} \<J_jv,Y\>J_iX,
\end{multline*}
for all $X,Y \in \Rn$, where $\{i,j,k\}=\{1,2,3\}$. Let for $a, b \in \Rn, \; a \wedge b$ be the skew-symmetric
operator defined by $(a \wedge b)X=\<a,X\>b-\<b,X\>a$. Then the above equation can be written in the form
$\n_YJ_k = \tfrac{1}{12} \ve_{ij}(\eta_i^{-1}-\eta_j^{-1}) (J_kv \wedge Y$ $+ v \wedge J_kY)
+\tfrac14 \eta_k^{-1}\sum\nolimits_{i \ne k} J_jv \wedge J_iY
+\sum\nolimits_{i \ne k} \tfrac12 (\eta_k-\eta_i)^{-1} \<J_jv,Y\>J_i$, that is,
\begin{equation}\label{eq:nyjx3w}
\begin{gathered}
\n_YJ_k = [J_k,AY], \quad
AY=\tfrac{1}{2} \sum\nolimits_{i=1}^3 \la_i J_iY \wedge J_iv + \sum\nolimits_{i=1}^3 \omega_i \<J_iv,Y\>J_i,\\
\la_i=\tfrac{1}{6} \ve_{jk}(\eta_j^{-1}-\eta_k^{-1}), \quad \omega_i=\tfrac{1}{4} \ve_{jk}(\eta_k-\eta_j)^{-1}
\quad \text{for } \{i,j,k\}=\{1,2,3\},
\end{gathered}
\end{equation}
where we used the fact that $[J_k, a \wedge b]=J_k a \wedge b + a \wedge J_k b$ and $[J_k,J_i]=2\ve_{ki}J_j$, for
$\{i,j,k\}=\{1,2,3\}$. By the Ricci formula, $\n^2_{Z,Y}J_k - \n^2_{Y,Z}J_k=[J_k,R(Y,Z)]$, where the tensor field
$\n^2 J_k$ is defined by $\n^2_{Z,Y}J_k=\n_Z(\n_YJ_k)-\n_{\n_ZY}J_k$ for vector fields $Y, Z$ on $\mathcal{U}'$. As
$\n_YJ_k = [J_k,AY]$ by \eqref{eq:nyjx3w}, this is equivalent to the fact that the operator
$F(Y,Z)=(\n_ZA)Y-(\n_YA)Z-[AY,AZ]-R(Y,Z)$ commutes with all the $J_s$'s, for all $Y,Z \in \Rn$ and all $s=1,2,3$.
As by \eqref{eq:confcs},
$R(Y, Z)  = Y \wedge \rho Z + \rho Y \wedge Z +\sum\nolimits_{i=1}^3 \eta_i (J_iY \wedge J_iZ + 2\< J_iY, Z \> J_i)$,
we obtain using \eqref{eq:nyjx3w} and the identities
$[ a \wedge b, c \wedge d]=\<a,d\> c \wedge b - \<a,c\> d \wedge b-$ $\<b,d\> c \wedge a + \<b,c\> d \wedge a$,
$[J_s, a \wedge b]=J_s a \wedge b + a \wedge J_s b$:
\begin{equation}\label{eq:FYZ}
\begin{aligned}
F(Y,Z) =& V(Y,Z) + \sum\nolimits_{i=1}^3 \<K_iY,Z\>J_i + S(Y,Z), \quad\text{where }
S(Y,Z) \in (\cI Y +\cI Z) \wedge \Rn \quad\text{and } \\
V(Y,Z)=& -\tfrac12 \sum\nolimits_{i=1}^3 \<J_iZ,Y\> (\la_i^2 v \wedge J_i v +
\ve_{jk}(\la_j\la_k -\la_i\la_k- \la_j\la_i) J_jv \wedge J_kv) \in \cI v \wedge \cI v,
\end{aligned}
\end{equation}
where for subspaces $L_1, L_2 \subset \Rn$, we denote $L_1 \wedge L_2$ the subspace of the space $\og(n)$ of the
skew-symmetric operators on $\Rn$ defined by $L_1 \wedge L_2=\Span(a \wedge b \, : \, a \in L_1, b \in L_2)$.
Note that if $L_1, L_2$ are $\cJ$-invariant (that is, $\cJ L_\a \subset L_\a$), then $L_1 \wedge L_2$
is $\mathrm{ad}_{\cJ}$-invariant, that is, $[J_s, L_1 \wedge L_2] \subset L_1 \wedge L_2$.

From \eqref{eq:nyjx3w} and using the fact that
$\omega_i \lambda_i=(24C)^{-1}\eta_i, \; \frac{d}{dt}\omega_i= 4\omega_i^2 + (12C)^{-1} \eta_i$ and
$\sum_i \omega_i^{-1}=0$, which follow from (\ref{eq:summarize}, \ref{eq:nyjx3w}), we obtain
\begin{equation}\label{eq:Ki}
K_i=-\omega_i((4\omega_i +\lambda_i) v \wedge J_i v + 4 \ve_{jk}(\omega_j+\omega_k) J_jv \wedge J_kv
+\lambda_i(48C+\|v\|^2)J_i +(J_iH+HJ_i)),
\end{equation}
where $\{i,j,k\}=\{1,2,3\}$, and $H$ is the symmetric operator associated to the Hessian of the function $t$
(that is, $\<HY,Z\>=Y(Zt)-(\n_Y)Zt$, for vector fields $Y, Z$ on $\mathcal{U}'$).


As $[F(Y,Z), J_s]=0$ and the subspace $\cI Y +\cI Z$ is $\cJ$-invariant (hence $(\cI Y +\cI Z) \wedge \Rn$
is $\mathrm{ad}_{\cJ}$-invariant), it follows from \eqref{eq:FYZ} that for all $Y, Z \in \Rn$ and all $s=1,2,3$,
\begin{equation}\label{eq:bracket}
[V(Y,Z),J_s] + \sum\nolimits_{i=1}^3 \<K_iY,Z\> [J_i,J_s] \in (\cI Y +\cI Z) \wedge \Rn.
\end{equation}
Take $Y, Z \in \cI v$ in \eqref{eq:bracket}. Then by
the $J^2$-property, $\cI Y +\cI Z = \cI v$ and  $[V(Y,Z),J_s] \in \cI v \wedge \cI v$, so \eqref{eq:bracket}
simplifies to $\sum\nolimits_{i \ne s} \ve_{is} \<K_iY,Z\> J_j \in \cI v \wedge \Rn$, where $\{i,j,s\}=\{1,2,3\}$.
Projecting this to the subspace $(\cI v)^\perp \wedge (\cI v)^\perp \subset \og(n)$ (with respect to the standard
inner product on $\og(n)$) and using the fact that
$(\cI v)^\perp$ is $\cJ$-invariant and $n \ge 8$, we get $\<K_iY,Z\>=0$, for all $i=1,2,3$ and all $Y,Z \in \cI v$.
Introduce the operators $\hat J_i= \pi_{\cI v} J_i \pi_{\cI v}, \; \hat H= \pi_{\cI v} H \pi_{\cI v}$ on $\cI v$.
As $\cI v$ is $\cJ$-invariant, the $\hat J_i$'s are anticommuting almost Hermitian structures on  $\cI v$.
Then the condition $\<K_iY,Z\>=0, \; Y,Z \in \cI v$, and \eqref{eq:Ki} imply
\begin{equation*}
(4\omega_i +\lambda_i) v \wedge \hat J_i v + 4 \ve_{jk}(\omega_j+\omega_k) \hat J_jv \wedge \hat J_kv
+\lambda_i(48C+\|v\|^2)\hat J_i +\hat J_i \hat H+ \hat H \hat J_i=0.
\end{equation*}
Multiplying by $\hat J_i$ and taking the trace we obtain $4\|v\|^2(\omega_i + \omega_j+\omega_k)
+\lambda_i(96C+3\|v\|^2) + \Tr \hat H=0$, where $\{i,j,k\}=\{1,2,3\}$, so $\lambda_i(96C+3\|v\|^2)$ does not
depend on $i=1,2,3$. As the $\la_i$'s are pairwise distinct (otherwise the condition $\sum_{i=1}^3 \eta_i^{-1}=0$
from \eqref{eq:summarize} is violated), we get $\|v\|^2=-32C$.

Now take $Y, Z \perp \cI v$ in \eqref{eq:bracket}. Projecting to $\cI v \wedge \cI v$ and using the fact that
$\cI v \wedge \cI v$ is $\mathrm{ad}_{\cJ}$-invariant we obtain that the operator
$V(Y,Z) + \sum\nolimits_{i=1}^3 \<K_iY,Z\> \hat J_i$ on $\cI v$ commutes with every $\hat J_s$. The centralizer
of the set $\{\hat J_1,\hat J_2,\hat J_3\}$ in the Lie algebra $\og(4)=\og(\cI v)$ is the three-dimensional subalgebra
spanned by $v \wedge \hat J_i v - \ve_{jk} \hat J_jv \wedge \hat J_kv, \; \{i,j,k\}=\{1,2,3\}$ (``the right
multiplication by the imaginary quaternions"). Substituting $V(Y,Z)$ from \eqref{eq:FYZ} and using the fact that
$\hat J_i=\|v\|^{-2}(v \wedge \hat J_i v + \ve_{jk} \hat J_jv \wedge \hat J_kv)$ we obtain that
the operator $V(Y,Z) + \sum\nolimits_{i=1}^3 \<K_iY,Z\> \hat J_i$ commutes with all the
$\hat J_s$'s, for $Y,Z \perp \cI v$, if and only if
$-\tfrac12 \<J_iZ,Y\> (\la_i^2 + \la_j\la_k -\la_i\la_k- \la_j\la_i) + 2 \|v\|^{-2} \<K_iY,Z\>=0$, for all
$i=1,2,3$. Substituting the $\la_i$'s from \eqref{eq:nyjx3w} and $\<K_iY,Z\>$ from \eqref{eq:Ki} and taking into
account that $\|v\|^2=-32C$, which is shown above, we obtain
$\<(J_iH+HJ_i-32C\la_i J_i)Y,Z\>=0$, for all $Y,Z \perp \cI v$ and all $i=1,2,3$.
Then $\pi (J_iH+HJ_i) \pi =32C\la_i \pi J_i \pi$, where $\pi=\pi_{(\cI v)^\perp}$. Multiplying both sides by
$\pi J_i \pi$ from the right and using the fact that $[\pi,J_i]=0$ (as $(\cI v)^\perp$ is $\cJ$-invariant) we get
$\pi (J_iHJ_i-H) \pi =-32C\la_i \pi$. Taking the traces of the both sides we obtain
$-2\Tr (\pi H \pi) =-32C\la_i (n-4)$, which is a contradiction, as $n > 4$ and the $\la_i$'s are pairwise distinct
(which follows from the equation $\sum_{i=1}^3 \eta_i^{-1}=0$ of \eqref{eq:summarize}).
\end{proof}
}

{

The next lemma shows that the relations similar to (\ref{eq:lnablarhoi}, \ref{eq:lnablarhoii}) of
Lemma~\ref{l:nablarho} also hold in all the remaining
cases when $n=8$ (that is, when $\nu \ne 3$ and when $\nu =3$ and $J_1J_2\ne \pm J_3$). As it is shown in
Lemma~\ref{l:locc1}, in all these cases the Weyl tensor has a smooth $\Cliff(7)$-structure in a
neighbourhood $\mU$ of every point $x \in M'$. Moreover, by assertion~2 of Lemma~\ref{l:formofR}, that
$\Cliff(7)$-structure is an ``almost Hermitian octonion structure", in the following sense. For every
$y \in \mU$, we can identify $\Ro=T_yM^8$ with $\Oc$ and of $\br^7$ with $\Oc'=1^\perp$ via linear isometries
$\iota_1, \iota_2$ respectively in such a way that the orthogonal multiplication \eqref{eq:ortmult} defined by
$\Cliff(7)$ has the form \eqref{eq:JuX8}: $J_uX = X u$, for every $X \in \Ro = \Oc, \; u \in \Oc'$.

\begin{lemma}\label{l:nablaJ8}
Let $x \in M' \subset M^8$ and let $\mU$ be the neigbourhood of $x$ defined in Lemma~\ref{l:locc1}. For every point
$y \in \mU$,
identify $\Ro=T_yM^8$ with $\Oc$ via a linear isometry in such a way that the Clifford structure $\Cliff(7)$ on $\Ro$
is given by \eqref{eq:JuX8}. Then there exist $m, t, b_{ij} \in \Ro = \Oc, \; i,j=1,\dots, 7$, such that for all
$X, U \in \Ro = \Oc$,
\begin{subequations}\label{eq:lnablaJ8}
\begin{gather}\label{eq:nUJX8}
    (\n_U J_i)X = \sum\nolimits_{j=1}^7 \<b_{ij},U\>Xe_j+(X(U^*m)-\<m,U\>X)e_i+\<m, Ue_i\>X,\\
    b_{ij}+b_{ji}=0, \label{eq:bijskew8} \\
    (\n_X \rho)U-(\n_U \rho)X=\tfrac34 (X \wedge U) t
    +2\sum\nolimits_{i=1}^7 \eta_i (\<me_i,U\>Xe_i-\<me_i,X\>Ue_i+2\<Xe_i,U\>me_i), \label{eq:nrho8}\\
    \n \eta_i=-4\eta_i m -\tfrac12 t.  \label{eq:neta8}
\end{gather}
\end{subequations}
\end{lemma}

\begin{proof}
In the proof we use standard identities of the octonion arithmetic (some of them are given in
Subsection~\ref{ss:cliffoct}).

By \cite[Lemma~7]{Nmm}, for the Clifford structure $\Cliff(7)$ given by \eqref{eq:JuX8}, there exist
$b_{ij} \in \Ro, \; i,j=1,\dots, 7$, satisfying \eqref{eq:bijskew8} and
an ($\br$-)linear operator $A: \Oc \to \Oc'$ such that for all $X, U \in \Ro = \Oc$,
\begin{equation}\label{eq:nUJX8mm}
    (\n_U J_i)X = \sum\nolimits_{j=1}^7 \<b_{ij},U\>Xe_j+(X\cdot AU)e_i + \<AU,e_i\>X.
\end{equation}
Equation \eqref{eq:confBZY} is a polynomial equation in $24$ real variables, the
coordinates of the vectors $X, Y, U \in \Ro$. It still holds, if we allow $X, Y, U$ to
be complex and extend the tensors $J_i, \n J_i$ and $\<\cdot,\cdot\>$ to $\bc^8$ by the complex linearity.
The complexified inner product $\<\cdot,\cdot\>$ takes values in $\bc$ and is a nonsingular quadratic form on
$\bc^8$. Moreover, equation \eqref{eq:JuX8} is still true, if we identify $\bc^8$ with the bioctonion algebra
$\OC$, and $\bc^7$ with $1^\perp=\Oc' \otimes \bc$, the orthogonal complement to $1$ in $\OC$.

Let $Y \in \OC$ be a nonzero isotropic vector (that is, $Y^*Y = 0$) and let
$\JC Y = \Span_{\mathbb{C}}(J_1Y, \ldots, J_7Y)$. Then $Y \in \JC Y$ and the space $\JC Y$ is
isotropic: the inner product of any two vectors from $\JC Y$ vanishes. Choose $X, U \in \JC Y$ and take the inner
product of the complexified equation \eqref{eq:confBZY} with $X$. As $X, Y$ and $U$ are mutually orthogonal, we get
\eqref{eq:confBZYX}, which further simplifies to $\sum_{i=1}^7 \eta_i \< J_iX, U\> \<(\n_YJ_i) Y,X\>=0$,
as $\|X\|^2=\|Y\|^2=\<J_iY,X\>=\< J_iY,U\>=0$. Using \eqref{eq:nUJX8mm} we obtain
$\sum_{i=1}^7 \eta_i \< J_iX, U\> \<(Y\cdot AY)e_i ,X\>=0$, for all isotropic vectors $Y$ and for all
$X, U \in \JC Y$. It follows that $Y\cdot AY \perp \sum_{i=1}^7 \eta_i \< J_iX, U\> Xe_i$, for all $X, U \in \JC Y$.
As $Y\cdot AY = J_{AY}Y \in \JC Y$ and $\JC Y$ is isotropic, we get $Y\cdot AY \perp \JC Y$, so
$Y\cdot AY \perp \JC Y + \Span_{\bc}(\{\sum_{i=1}^7 \eta_i  \<J_i X, U\> J_iX \;|\; X, U \in \JC Y\})$.
Following the arguments in the proof of \cite[Lemma~8]{Nmm} starting with formula (29), we obtain that
$AU = U^*m - \<U, m\> \, 1$, for some $m \in \Oc$. Then equation \eqref{eq:nUJX8} follows from \eqref{eq:nUJX8mm}.

To prove \eqref{eq:nrho8} and \eqref{eq:neta8}, introduce the vectors $f_{ij} \in \Ro, \; i,j=1, \dots, 8$, and the
quadratic map $T: \Ro \to \Ro$ (similar to the map $Q$ of \eqref{eq:defQ}) by
\begin{gather}\label{eq:fij8}
f_{ij}=(\eta_i-\eta_j)b_{ij}+\K_{ij}(\n \eta_i-2\eta_i m), \\ \label{eq:T8}
\<T(X), U\>=\tfrac13 \<(\n_X \rho)U-(\n_U \rho) X,X\>-\sum\nolimits_{i=1}^7 \eta_i \<me_i,X\>\<Xe_i, U\>.
\end{gather}
Note that $f_{ij}=f_{ji}$ and $\<T(X), X\> = 0$. Take $X, Y, U$ to be mutually orthogonal vectors in $\Ro$. By
\eqref{eq:nUJX8} and \eqref{eq:bijskew8},
$\<(\n_UJ_i)X, Y \>=\sum\nolimits_{j=1}^7 \<b_{ij},U\>\<Xe_j, Y \>-\<m,U\>\<Xe_i, Y \>+\<(X(U^*m))e_i, Y \>=
\sum\nolimits_{j=1}^7 \<b_{ij}-\K_{ij} m,U\>\<Xe_j, Y \>+\<m((e_iY^*)X), U \>$, so
every term on the left-hand side
of \eqref{eq:confBZYX} can be written as the inner product of a vector depending on $X$ and $Y$ by $U$. As
$U \perp X, Y$ is arbitrary, we find after substituting \eqref{eq:JuX8} and \eqref{eq:nUJX8} into \eqref{eq:confBZYX}
and rearranging the terms:
\begin{equation*}
\begin{split}
&\|X\|^2 T(Y) +\|Y\|^2 T(X) + 2\sum\nolimits_{i=1}^7 \eta_i \<Ye_i,X\> (m((e_iY^*)X)+(Y(X^*m))e_i) \\
&+\sum\nolimits_{i,j=1}^7 \<Ye_j,X\>(\<f_{ij},X\> Ye_i-\<f_{ij},Y\> Xe_i)
- \sum\nolimits_{i,j=1}^7 \<Ye_i,X\>\<Ye_j,X\> f_{ij} \in \Span(X,Y),
\end{split}
\end{equation*}
for all $X \perp Y$ (where we used the fact that $(X(Y^*m))e_i=-(Y(X^*m))e_i$, as $X \perp Y$). Taking the inner
products with $X$ and with $Y$ we obtain
\begin{equation*}
\begin{split}
&\|X\|^2 T(Y) +\|Y\|^2 T(X) + 2\sum\nolimits_{i=1}^7 \eta_i \<Ye_i,X\> (m((e_iY^*)X)+(Y(X^*m))e_i) \\
&+\sum\nolimits_{i,j=1}^7 \<Ye_j,X\>(\<f_{ij},X\> Ye_i-\<f_{ij},Y\> Xe_i)
- \sum\nolimits_{i,j=1}^7 \<Ye_i,X\>\<Ye_j,X\> f_{ij}\\
&=\<T(Y),X\>X+\<T(X),Y\>Y,
\end{split}
\end{equation*}
for all $X \perp Y$. Taking $X=Yu, \; u= \sum_{i=1}^7 u_ie_i \in \Oc'$ and regrouping the terms we obtain
\begin{equation} \label{eq:XYu8}
\begin{split}
&\|u\|^2 T(Y) + T(Yu)
+ 2\sum\nolimits_{i=1}^7 \eta_i u_i (2\<Y,me_i\> Yu - 2\<Yu,me_i\> Y +
2\|Y\|^2(mu)e_i) \\
&+\sum\nolimits_{i,j=1}^7 u_j(\<f_{ij}+8\K_{ij}\eta_im,Yu\> Ye_i-\<f_{ij}+8\K_{ij}\eta_im,Y\> (Yu)e_i)
- \sum\nolimits_{i,j=1}^7 \|Y\|^2 u_i u_j f_{ij}\\
&=\|Y\|^{-2} \<T(Y),Yu\>Yu+\|Y\|^{-2}\<T(Yu),Y\>Y,
\end{split}
\end{equation}
where we used 
$m((e_iY^*)X)+(Y(X^*m))e_i=2\<Y,me_i\> Yu - 2\<Yu,me_i\> Y +4 \<Yu,m\>Ye_i-4 \<Y,m\>(Yu)e_i+2\|Y\|^2(mu)e_i$, which
follows from $m((e_iY^*)X)=(Y(X^*m))e_i-2\<m,Ye_i\> X - 2\<X,me_i\> Y$, for all $X,Y$, and
$(Y(X^*m))e_i=-2 \<Y,m\>(Yu)e_i-2 \<Y,mu\>Ye_i+\|Y\|^2(mu)e_i$, for $X=Yu, \; u \perp 1$. By assertion~1 of
Lemma~\ref{l:level} (with $\nu=1$ and $\cI Y = \Span(Y, Yu)$) we obtain that both coefficients on the right-hand side
of \eqref{eq:XYu8}, $\|Y\|^{-2}\<T(Y),Yu\>$ and $\|Y\|^{-2}\<T(Yu),Y\>$, are linear forms of $Y \in \Ro$, for every
$u \in \Oc'$. As $\<T(Y),Y\>=0$, this implies that there exists an ($\br$-)linear operator $C: \Oc \to \Oc'$ such that
$\|Y\|^{-2}Y^*T(Y)=CY$, so $T(Y)=Y \cdot CY$, for all $Y \in \Oc$.
Substituting this to \eqref{eq:XYu8} and rearranging the terms we obtain
\begin{equation} \label{eq:XYu8C}
\begin{split}
&(Yu) \Bigl(C(Yu) -\sum\nolimits_{i,j=1}^7 u_j\<f_{ij}+8\K_{ij}\eta_im,Y\>e_i\Bigr)\\
&+ Y\Bigl(\|u\|^2 CY + 4\sum\nolimits_{i=1}^7 \eta_i u_i (\<Y,me_i\> u - \<Yu,me_i\> 1 +Y^*((mu)e_i)) \\
&+\sum\nolimits_{i,j=1}^7 u_j\<f_{ij}+8\K_{ij}\eta_im,Yu\> e_i
- \sum\nolimits_{i,j=1}^7 u_i u_j Y^*f_{ij} - \<CY,u\>u+\<C(Yu),u\>1 \Bigr)= 0,
\end{split}
\end{equation}
The left-hand side of \eqref{eq:XYu8C} has the form $(Yu)L(Y,u)+YF(Y,u)$, where $L(Y,u)$ and $F(Y,u)$ are
($\br$-) linear operator on $\Oc$, for every $u \in \Oc'$. By \cite[Lemma~6]{Nmm}, for every unit octonion
$u \in \Oc', \; L(Y,u)=\<a(u),Y\>1+\<t(u),Y\>u+Y^*p(u)$, for some functions $a, t, p: S^6 \subset \Oc' \to \Oc$.
Extending $a, t, p$ by homogeneity (of degree $1,0,1$ respectively) to $\Oc'$ we obtain 
$C(Yu) -\sum\nolimits_{i,j=1}^7 u_j\<f_{ij}+8\K_{ij}\eta_im,Y\>e_i =\<a(u),Y\>1+\<t(u),Y\>u+Y^*p(u)$,
for all $u \in \Oc'$. Moreover, $p(u) = -a(u)$, as $C(Y) \perp 1$.
By the linearity of the left-hand side by $u$, we get $\<a(u_1+u_2)-a(u_1)-a(u_2),Y\>1+
\<t(u_1+u_2)-t(u_1),Y\>u_1+\<t(u_1+u_2)-t(u_2),Y\>u_2+Y^*(a(u_1+u_2)-a(u_1)-a(u_2))=0$, for all $u_1, u_2 \in \Oc'$.
Then $Y^*(a(u_1+u_2)-a(u_1)-a(u_2)) \in \Span(1, u_1, u_2)$, for all $Y \in \Oc$, which is only possible when $a(u)$
is linear, that is $a(u)=Bu$, for some ($\br$-)linear operator $B: \Oc' \to \Oc$. It follows that
$t(u_1+u_2)=t(u_1)=t(u_2)$, that is, $t \in \Oc$ is a constant. So
$C(Yu) =\sum\nolimits_{i,j=1}^7 u_j\<f_{ij}+8\K_{ij}\eta_im,Y\>e_i+\<Bu,Y\>1+\<t,Y\>u-Y^*Bu$.
Taking the inner product of the both sides with $v \in \Oc'$ and subtracting from the resulting equation the same
equation with $u$ and $v$ interchanged we obtain $\<C(Yu),v\>-\<C(Yv),u\>=\<Bv,Yu\>-\<Bu,Yv\>$, since $f_{ij}=f_{ji}$
by \eqref{eq:fij8}. It follows that $\<C^tv-Bv, Yu\>=\<C^tu-Bu, Yv\>$, where $C^t$ is the operator adjoint to $C$.
Now taking $u \perp v$ and $Y=uv$
we get $\|u\|^2 \<C^tv-Bv, v\>=-\|v\|^2 \<C^tu-Bu, u\>$, which implies $C = B^t$. Then from the above,
$\<C(Yu),e_i\> = \sum\nolimits_{j=1}^7 u_j\<f_{ij}+8\K_{ij}\eta_im,Y\>+\<t,Y\>u_i-\<Bu,Ye_i\>= \<Be_i, Yu\>$, so
$\sum\nolimits_{j=1}^7 u_j(f_{ij}+\K_{ij}(8\eta_im+t))+(Bu)e_i+(Be_i)u=0$. Therefore
\begin{equation} \label{eq:fT8}
T(Y)= Y \cdot CY = Y \cdot B^tY, \qquad
f_{ij}=-\K_{ij}(8\eta_i m +t) - (Be_i)e_j-(Be_j)e_i.
\end{equation}
Substituting \eqref{eq:fT8} to  \eqref{eq:XYu8C} and simplifying we obtain
$-\<Lu \cdot u, Y\>Y - \<Lu, Y\> Yu + \|Y\|^2 Lu \cdot u=0$, where $Lu=4Bu-tu-4\sum\nolimits_{i=1}^7 \eta_iu_i m e_i$.
Taking $Y \perp Lu, Lu \cdot u$ we get $Lu = 0$, so
\begin{equation}\label{eq:A8}
Bu=\tfrac14 tu+\sum\nolimits_{i=1}^7 \eta_iu_i m e_i.
\end{equation}
Substituting \eqref{eq:A8} to the first equation of \eqref{eq:fT8} and then to \eqref{eq:T8} and simplifying we obtain
that for arbitrary $X, U \in \Oc$,
$\<(\n_X \rho)U-(\n_U \rho)X, X\>=\frac34 (\<t,X\>\<X,U\>-\|X\|^2 \<t,U\>)+
6\sum\nolimits_{i=1}^7 \eta_i\<Xe_i,U\> \<m e_i,X\>$. Polarizing this equation we get
\begin{align*}
\<(\n_Y \rho)U-(\n_U \rho)Y, X\>+\<(\n_X \rho)U-(\n_U \rho)X, Y\>&=
\tfrac34 (\<t,X\>\<Y,U\>+\<t,Y\>\<X,U\>-2\<X,Y\> \<t,U\>)\\
&+6\sum\nolimits_{i=1}^7 \eta_i(\<Xe_i,U\> \<m e_i,Y\>+\<Ye_i,U\> \<m e_i,X\>).
\end{align*}
Subtracting the same equation, with $X$ and $U$ interchanged and using the fact that $\rho$ is symmetric we get
\eqref{eq:nrho8}. The second equation of \eqref{eq:fT8} and \eqref{eq:A8} give $f_{ii}=-6\eta_i m -\frac12 t$, which
by \eqref{eq:fij8} implies \eqref{eq:neta8}.
\end{proof}

\begin{lemma}\label{l:codazzi}
In the assumptions of Theorem~\ref{t:cocl}, let $x \in M'$, where $M' \subset M^n$ is defined in Lemma~\ref{l:locc1}.
Then there exists a neighbourhood $\mathcal{U}=\mathcal{U}(x)$ and a smooth metric on $\mathcal{U}$ conformally
equivalent to the original metric whose curvature tensor has the form \eqref{eq:confcs}, with $\rho$ a multiple
of the identity.
\end{lemma}

\begin{proof}
Let $x \in M'$ and let $\mathcal{U}$ be the neighbourhood of $x$ on which the Weyl tensor has the smooth Clifford
structure defined in Lemma~\ref{l:locc1}. We can assume that $\nu >0$, as in the case of a $\Cliff(0)$-structure,
the curvature tensor given by \eqref{eq:confcs} has the form
$R(X, Y) Z = \< X, Z \> \rho Y +\<\rho X, Z\> Y - \< Y, Z \> \rho X -\<\rho Y, Z\> X$, so the Weyl tensor vanishes.
Then the metric on $\mU$ is locally conformally flat, that is, is conformally equivalent to a one with $\rho=0$.

If $n=8, \; \nu =7$, and all the $\eta_i$'s at $x$ are equal, then they are equal at some neighbourhood of $x$ (by
definition of $M'$). By Remark~\ref{rem:n8nu7}, we can replace $\rho$ by $\rho +\frac32 \eta_1 \, \id$ and
$\eta_i$ by $0=\eta_i-\eta_1$ in \eqref{eq:confcs} arriving at the case $\nu=0$ considered above.

For the remaining part of the proof, we will assume that in the case $n=8, \; \nu =7$, at least two of the $\eta_i$'s
at $x$ are different; up to relabelling, let $\eta_1 \ne \eta_2$ at $x$, and also on a neighbourhood of $x$ (replace
$\mathcal{U}$ by a smaller neighbourhood, if necessary).
Let $f$ be a smooth function on $\mathcal{U}$ and let $\<\cdot, \cdot\>' = e^f\<\cdot, \cdot\>$. Then
$W'=W, \; J_i'=J_i, \; \eta_i'= e^{-f} \eta_i$ and, on functions, $\n'=e^{-f}\n$, where we use the dash for the
objects associated to metric $\<\cdot, \cdot\>'$. Moreover, the curvature tensor $R'$ still has the form
\eqref{eq:confcs}, and all the identities of Lemma~\ref{l:nablarho} and of Lemma~\ref{l:nablaJ8} remain valid.

In the cases considered in Lemma~\ref{l:nablarho}, the ratios $\eta_i/\eta_1$ are constant, as it follows from
(\ref{eq:neta},\ref{eq:allequal}).
In particular, taking $f=\ln|\eta_1|$ we obtain that $\eta_1'$ is a constant,
so all the $\eta_i'$ are constant, $m_i'=0$ by \eqref{eq:neta}, so $(\n'_Y \rho')U-(\n'_U \rho')Y=0$
by \eqref{eq:skewrho}. In the case $n=8, \; \nu =7$ (Lemma~\ref{l:nablaJ8}), take $f=\ln|\eta_1-\eta_2|$. Then
by \eqref{eq:neta8}, $\n f = - 4m$ and $\n'\eta_i'=-\tfrac12 e^{-2f}t$ which implies
$m'=-\frac14 \n'\ln|\eta_1'-\eta_2'|=0, \; t'=e^{-2f}t$, again by \eqref{eq:neta8} for the metric $\<\cdot, \cdot\>'$.
Then by \eqref{eq:nrho8}, $(\n'_X \rho')U-(\n'_U \rho')X=\tfrac34 (X \wedge' U) t'$.
By Remark~\ref{rem:n8nu7}, we can replace $\rho'$ by $\tilde\rho=\rho' +\frac32 (\eta_1' +C)\, \id$ and
$\eta_i'$ by $\tilde\eta_i=\eta_i'-(\eta_1' +C)$ without changing the curvature tensor $R'$ given by \eqref{eq:confcs}
($C$ is a constant chosen in such a way that $\tilde\eta_i \ne 0$ anywhere on $\mathcal{U}$). Then
by \eqref{eq:nrho8} and \eqref{eq:neta8} for the metric
$\<\cdot, \cdot\>', \; (\n'_X \tilde\rho)U-(\n'_U \tilde\rho)X=0$.

Dropping the dashes and the tildes, we obtain that, up to a conformal smooth change of the metric on
$\mathcal{U}$, the curvature tensor has the form \eqref{eq:confcs}, with $\rho$ satisfying the identity
\begin{equation*}
 (\n_Y \rho)X=(\n_X \rho)Y,
\end{equation*}
for all $X, Y$, that is, with $\rho$  being a symmetric \emph{Codazzi tensor}.

Then by \cite[Theorem~1]{DS}, at every point of $\mathcal{U}$, for any three eigenspaces
$E_\b, E_\gamma, E_\a$ of $\rho$, with $\a \notin \{\b, \gamma\}$,
the curvature tensor satisfies $R(X,Y)Z = 0$, for all $X \in E_\b, \; Y \in E_\gamma, \; Z \in E_\a$.
It then follows from \eqref{eq:confcs} that
\begin{equation}\label{eq:codazzi}
\begin{gathered}
\sum\nolimits_{i=1}^\nu \eta_i (2 \< J_iX, Y \> J_iZ + \< J_iZ, Y \> J_iX - \< J_iZ, X \> J_iY) = 0, \\
\text{for all } X \in E_\b, \; Y \in E_\gamma, \; Z \in E_\a, \quad \a \notin \{\b, \gamma\}.
\end{gathered}
\end{equation}
Suppose $\rho$ is not a multiple of the identity. Let $E_1, \ldots, E_p, \; p \ge 2$, be the eigenspaces of $\rho$.
If $p>2$, denote $E_1'=E_1, \; E_2'=E_2 \oplus \dots \oplus E_p$. Then by linearity, \eqref{eq:codazzi}
holds for any $X, Y \in E'_\a, \; Z \in E'_\b$, such that $\{\a, \b\}=\{1,2\}$. Hence to prove
the lemma it suffices to show that \eqref{eq:codazzi} leads to a contradiction,
in the assumption $p=2$. For the rest of the proof, suppose that $p=2$. Denote $\dim E_\a =d_\a, \; d_1 \le d_2$.

Choose
$Z \in E_\a, \; X, Y \in E_\b,\; \a \ne \b$, and take the inner product of \eqref{eq:codazzi} with $X$. We get
$\sum_{i=1}^\nu \eta_i \< J_iX, Y \>\< J_iX, Z\> = 0$.
It follows that for every $X \in E_\a$, the subspaces $E_1$ and $E_2$ are invariant subspaces of the symmetric
operator $\hat R_X \in \End(\Rn)$ defined by $\hat R_XY=\sum_{i=1}^\nu \eta_i \< J_iX, Y \>J_iX$. So
$\hat R_X$ commutes with the orthogonal projections $\pi_\b: \Rn \to E_\b, \; \b=1,2$. Then for all $\a,\b=1,2$
($\a$ and $\b$ can be equal), all $X \in E_\a$ and all $Y \in \Rn, \quad
\sum_{i=1}^\nu \eta_i \< J_iX, \pi_\b Y \> J_iX = \sum_{i=1}^\nu \eta_i \< J_iX, Y \> \pi_\b J_iX$.
Taking $Y=J_jX$ we get that $\pi_\b J_jX \subset \cJ X$, that is $\pi_\b \cJ X \subset \cJ X$, for all
$X \in E_\a, \; \a,\b=1,2$. As $\pi_1 + \pi_2 = \id$, we obtain
$\cJ X \subset \pi_1 \cJ X \oplus \pi_2 \cJ X \subset \cJ X$, hence $\cJ X = \pi_1 \cJ X \oplus \pi_2 \cJ X$.
As every function $f_{\a\b}: E_\a \to \mathbb{Z}, \; \a,\b=1,2$, defined by
$f_{\a\b}(X) = \dim \pi_\b \cJ X, \; X \in E_\a$, is lower semi-continuous, and $f_{\a 1}(X) + f_{\a 2}(X) = \nu$
for all nonzero $X \in E_\a$, there exist constants $c_{\a\b}$, with $c_{\a 1} + c_{\a 2} = \nu$, such that
$\dim \pi_\b \cJ X = c_{\a\b}$, for all $\a,\b=1,2$ and all nonzero $X \in E_\a$.

Let $X, Y \in E_\a, \; Z \in E_\b, \; \b \ne \a$. Taking the inner product of \eqref{eq:codazzi} with
$J_jZ, \; j=1, \ldots, \nu$, we get
\begin{equation*}
2 \eta_j \< J_jX, Y \> \|Z\|^2 =
\sum\nolimits_{i \ne j} \eta_i (\< J_iZ, X \> \<J_iY,J_jZ\> - \< J_iZ, Y \> \<J_iX, J_jZ\>).
\end{equation*}
As $\< J_iZ, X \>=\< J_i\pi_\b Z, X \>=-\< Z, \pi_\b J_iX \>$ (and similarly for $\<J_iZ, Y \>$), the right-hand side,
viewed as a quadratic form of $Z\in E_\b$, vanishes for all $Z \in (\pi_\b \cJ X)^\perp \cap (\pi_\b \cJ Y)^\perp$,
that is, on a subspace of dimension at least $d_\b-2 c_{\a\b}$. So for $\a \ne \b$, either $2 c_{\a\b} \ge d_\b$,
or $\cJ E_\a \perp E_\a$, that is, $\pi_\b \cJ X= \cJ X$, for all $X \in E_\a$, so $c_{\a\b}=\nu$.

Similarly, if $Z \in E_\a, \; X, Y \in E_\b, \; \b \ne \a$, the inner product of \eqref{eq:codazzi} with
$J_jX, \; j=1, \ldots, \nu$, gives
\begin{equation*}
\eta_j \< J_jZ, Y \> \|X\|^2 =
\sum\nolimits_{i=1}^\nu \eta_i (- 2 \< J_iX, Y \> \<J_iZ, J_jX\> + \< J_iZ, X \> \<J_iY, J_jX\>).
\end{equation*}
As $\< J_iX, Y \>=-\< X, \pi_\b J_iY \>, \; \< J_iZ, X \>=-\< X, \pi_\b J_iZ \>$, the right-hand side, viewed as a
quadratic form of $X\in E_\b$, vanishes on the subspace $(\pi_\b \cJ Y)^\perp \cap (\pi_\b \cJ Z)^\perp$ whose
dimension is at least $d_\b-c_{\a\b} - c_{\b\b}$. We obtain that for $\a \ne \b$, either
$c_{\a\b} + c_{\b\b} \ge d_\b$, or $\cJ E_\a \perp E_\b$, that is, $\pi_\b \cJ Z= 0$, for all $Z \in E_\a$, so
$c_{\a\b}=0$. As the equation $c_{\a\b}=0$ contradicts both $2 c_{\a\b} \ge d_\b$ and $c_{\a\b}=\nu$ (as $\nu >0$), we
must have $c_{\a\b} + c_{\b\b} \ge d_\b$. Then $2 \nu = \sum_{\a\b} c_{\a\b} \ge d_1+d_2 = n$.

This proves the lemma in all the cases when $2 \nu < n$, that is,
in all the cases except for $n=8, \; \nu \ge 4$ (as it follows from Lemma~\ref{l:nvsnu}).

Consider the case $n=8$. We identify $\Ro$ with $\Oc$ and assume that the $J_i$'s act as in \eqref{eq:JuX8}. Let
$D: \Oc \to \Oc$ be the symmetric operator defined by $D1=0, \; De_i = \eta_i e_i$. By \eqref{eq:8sq},
condition \eqref{eq:codazzi} still holds if we replace $D$ by $D + c \, \mathrm{Im}$, where $\mathrm{Im}$ is the
operator of taking the imaginary part of an octonion. So we can assume that the eigenvalue of the maximal multiplicity
of $D_{|\Oc'}$ is zero (one of them, if there are more than one). Then in \eqref{eq:codazzi}, $\nu = \rk \, D$.
By construction, $\nu \le 6$, and we only need to consider the cases when $\nu \ge 4$, as it is shown above.

By \eqref{eq:JuX8}, $\< J_iX, Y \> J_iZ= \<Xe_i, Y \> Ze_i=\<e_i, X^*Y \> Ze_i$,
so $\sum_{i=1}^\nu \eta_i \< J_iX, Y \> J_iZ=\sum_{i=1}^\nu \eta_i \<e_i, X^*Y \> Ze_i$ $=
\sum_{i=1}^7 \<De_i, X^*Y \> Ze_i=ZD(X^*Y)$, as $D$ is symmetric and $D1=0$. Then \eqref{eq:codazzi} can be rewritten
as
\begin{equation}\label{eq:codazzi8}
2ZD(X^*Y)+XD(Z^*Y)-YD(Z^*X)= 0, \qquad \text{for all } X,Y \in E_\b, \; Z \in E_\a, \; \a \ne \b.
\end{equation}
Taking the inner product of \eqref{eq:codazzi8} with $X$ (and using the fact that $D$ is symmetric, $D1=0$ and
$Y^*X=2\<X,Y\>1-X^*Y$) we obtain $\<D(X^*Y),X^*Z\>=0$. It follows that for every $X \in E_\b$, the subspaces $E_1$ and
$E_2$ are invariant subspaces of the symmetric operator $L_XDL_X^t$, where $L_X:\Oc \to \Oc$ is the left
multiplication by $X$ (note that $L_{X^*}=L_X^t$ and that $L_XDL_X^t$ coincides with the operator $\hat R_X$
introduced above). So $L_XDL_X^t$ commutes with the both orthogonal projections $\pi_\a: \Ro \to E_\a, \; \a=1,2$.
It follows that for every $\a, \b$ (not necessarily distinct) and every $X \in E_\b$, the operator $D$ commutes
with $L_X^t\pi_\a L_X=\|X\|^2 \pi_{X^*E_\a}$, that is,
\begin{equation}\label{eq:[Dpi]}
\text{the space $X^*E_\a$ is an invariant subspace of $D$, for all $\a, \b$, and all $X \in E_\b$}.
\end{equation}
Consider all the possible cases for the dimensions $d_\a$ of the subspaces $E_\a$.

Let $(d_1, d_2)=(1,7)$, and let $u$ be a nonzero vector in $E_1$. Then by \eqref{eq:[Dpi]}, every line spanned by
$X^*u, \; X \perp u$ (that is, every line in $\Oc'$) is an invariant subspace of $D$. It follows that
$D_{|\Oc'}$ is a multiple of the identity, which is a contradiction, as $\rk \, D = \nu, \; 4 \le \nu \le 6$.

Let $(d_1, d_2)=(2,6)$, and let $E_1=\Span(u, ue), \; e \in \Oc', \; \|e\|=\|u\|=1$. Then $E_2=uL$, where
$L=\Span(1, e)^\perp$. By \eqref{eq:[Dpi]} with $E_\a=E_1$ and $X=uU^*=-uU \in E_2, \; U \in L$,
every two-plane $\Span(U, (Uu^*)(ue)), \; U \in L$, is an invariant
subspace of $D$. Note that $(Uu^*)(ue) \in L$, for all $U \in L$, and moreover, the operator $J$ defined by
$JU=(Uu^*)(ue)$ is an almost Hermitian structure on $L$. Then $L$ is an invariant subspace of $D$ (as the sum of the
invariant subspaces $\Span(U, JU), \; U \in L$) and $JD_{|L}U \in \Span(U, JU)$, for all $U \in L$ (as $\Span(U, JU)$
is both $J$- and $D_{|L}$-invariant). From assertion~1 of Lemma~\ref{l:level} it follows that the operator $JD_{|L}$
is a linear combination of $\id_{|L}$ and $J$. As $D$ is symmetric and its eigenvalue of the maximal multiplicity is
zero, $D_{|L}=0$. Then $\nu = \rk \, D \le 1$, which is a contradiction.

For the cases $(d_1, d_2)=(3,5), \, (4,4)$, we use the notion of the Cayley plane. A four-dimensional subspace
$\cp \subset \Oc$ is called a \emph{Cayley plane}, if for orthonormal octonions $X, Y, Z \in \cp$,
$X(Y^*Z) \in \cp$. This definition coincides with \cite[Definition~IV.1.23]{HL}, if we disregard the
orientation. We will need the following properties of the Cayley plane (they can be found in \cite[Section~IV]{HL}
or proved directly):
\begin{enumerate}[(i)]
  \item \label{it:cp1}
  A Cayley plane is well-defined; moreover, if $X(Y^*Z) \in \cp$ for some triple $X, Y, Z$ of orthonormal octonions in
  $\cp$, then the same is true for any triple $X, Y, Z \in \cp$ (possibly, non-orthonormal).
  \item \label{it:cp2}
  If $\cp$ is a Cayley plane, then the subspace $X^*\cp$ is the same for all nonzero $X \in \cp$; we call this
  subspace $\cp^*\cp$.
  \item \label{it:cp3}
  If $\cp$ is a Cayley plane, then $\cp^\perp$ is also a Cayley plane and $\cp^{\perp*}\cp^\perp=\cp^*\cp$.
  Moreover, for all nonzero $X \in \cp^\perp$, the subspace $X^*\cp$ is the same and is equal to $(\cp^*\cp)^\perp$.
  \item
  For every nonzero $e\in \Oc$ and every pair of orthonormal imaginary octonions $u, v$, the subspace
  $\cp=\Span(e,eu,ev,(eu)v)$ is a Cayley plane; every Cayley plane can be obtained in this way.
\end{enumerate}

Let $(d_1, d_2)=(3,5)$. Then $E_1$ is contained in a Cayley plane $\cp$ (spanned by $E_1$ and $X(Y^*Z)$, for some
orthonormal vectors $X, Y, Z \in E_1$), so $\cp^\perp \subset E_2$. Let $U$ be a unit vector in the orthogonal
complement to $\cp^\perp$ in $E_2$. Then for every nonzero $X \in \cp^\perp$, $X^*E_2 = \cp^*\cp \oplus \br(X^*U)$, by
properties (\ref{it:cp2}, \ref{it:cp3}). As for any two invariant subspaces of a symmetric operator, their
intersection and the orthogonal complements to it in each of them are also invariant, it follows from \eqref{eq:[Dpi]}
that both $\cp^*\cp$ and every line $\br(X^*U), \; X \in \cp^\perp$, are invariant subspaces of $D$. Then the
restriction of $D$ to the four-dimensional space $(\cp^\perp)^*U$ is a multiple of the identity on that space. As the
eigenvalue of the maximal multiplicity of $D$ is zero, $\br 1 \oplus (\cp^\perp)^*U \subset \Ker D$. Then
$\nu = \rk \, D \le 3$, which is again a contradiction.

Let now $d_1=d_2=4$. First assume that $E_1$ is not a Cayley plane. Let $X_1, X_2$ be orthonormal vectors
in $E_1$. Then $X^*_1E_1 \cap X^*_2E_1 \supset \Span(1, X^*_1X_2)$, as $X^*_2X_1=-X^*_1X_2$. Moreover, for any unit
vector $Y \in X^*_1E_1 \cap X^*_2E_1$ orthogonal to $\Span(1, X^*_1X_2)$ we have $Y=X^*_1X_3=X^*_2X_4$ for some
$X_3, X_4 \in E_1$, $X_3, X_4 \perp X_1, X_2$, which implies $X_2(X^*_1X_3)=X_4 \in E_1$, so $E_1$ is a Cayley plane
by property \eqref{it:cp1}. It follows that $X^*_1E_1 \cap X^*_2E_1 = \Span(1, X^*_1X_2)$. As by \eqref{eq:[Dpi]} both
subspaces on the left-hand side are invariant under $D$ and as $\br 1$ is an invariant subspace of $D$, we obtain that
every line $\br(X^*_1X_2), \; X_1, X_2 \in E_1$ is an invariant subspace of $D$ (that is, $X^*_1X_2$ is an eigenvector
of $D$). Then the space $L=\Span(E_1^*E_1)$ lies in an eigenspace of $D$, so $D_{|L}$ is a multiple of $\id_{|L}$. If
$X_1,X_2,X_3 \in E_1$ are orthonormal, then $X_2^*X_3 \notin X_1^*E_1$, as $E_1$ is not a Cayley plane. So
$\dim L \ge 5$. As the eigenvalue of the maximal multiplicity of $D$ is zero, $\nu = \rk \, D \le 3$, a contradiction.

Let again $d_1=d_2=4$, and let $E_1$ be a Cayley plane. Then $E_2=(E_1)^\perp$ is also a Cayley plane by
property \eqref{it:cp3}. Moreover, by the same property, $E_1^*E_1=E_2^*E_2=V_1$ and $E_1^*E_2=E_1^*E_2=V_2$, where
$V_1, V_2$ are mutually orthogonal four-dimensional subspaces of $\Oc$, and $1 \in V_1$. From \eqref{eq:[Dpi]},
each of the two spaces $V_1, V_2$ is invariant under $D$. Let $X, Y \in E_1, \; Z, W \in E_2$, with $X, Z \ne 0$, and
let $u=X^{-1}Y, \; v=Z^{-1}W$. As $X^{-1}=\|X\|^{-2}X^*$, $L_{X^{-1}}E_1=V_1$ by property \eqref{it:cp2}.
Similarly, $L_{Z^{-1}}E_2=V_1$. Taking the inner product of \eqref{eq:codazzi8}
with $W$ we obtain that for all $X \in E_1, \; Z \in E_2, \; u,v \in V_1$,
\begin{equation*}
2\|Z\|^2\|X\|^2 \<Du,v\>-\<D(Z^*(Xu)),Z^*(Xv)\>=-\<D(Z^*X),Z^*((Xu)v)\>.
\end{equation*}
The left-hand side is symmetric in $u, v$. As $(Xu)v=-(Xv)u$, for any $u \perp v, \; u, v \perp 1$, we
obtain $\<D(Z^*X),Z^*((Xu)v)\>=0$ for all $u, v \in V_1, \; u \perp v, \; u, v \perp 1$, and all
$X \in E_1, \; Z \in E_2$. Given any nonzero orthogonal $X, X' \in  E_1$, we can find
$u, v \in V_1, \; u \perp v, \; u, v \perp 1$, such that $X'=(Xu)v$. To see that note that for every
$u \in V_1=E_1^*E_1, \; Xu \in E_1$ by property \eqref{it:cp1}. As $L_X$ is nonsingular, $L_X(V_1 \cap 1^\perp)$ is a
three-dimensional subspace of $E_1$. The same is true with $X$ replaced by $X'$. Therefore, for some
$u, v \in V_1 \cap 1^\perp, \; Xu=X'v$, hence $X'=-\|v\|^{-2}(Xu)v$. As $X' \perp X$, we get $\<X,(Xu)v\>=0$, so
$u \perp v$. Thus $\<D(Z^*X),Z^*X'\>=0$, for any $Z \in E_2$ and any orthogonal $X, X' \in E_1$. As
$Z^*E_1 = V_2$, for any nonzero $Z \in E_2$, by properties \ref{it:cp2}, \ref{it:cp2}), and the operator
$L_{Z^*}$ is orthogonal when $\|Z\|=1$ we get $\<Dv_1,v_2\>=0$, for any two orthogonal vectors $v_1, v_2 \in V_2$.
It follows that the restriction of $D$ to its invariant subspace $V_2$ is a multiple of the identity. As
$V_2 \subset \Oc'$ and the eigenvalue of $D_{|\Oc'}$ of the maximal multiplicity is zero we obtain that
$\br 1 \oplus V_2 \subset \Ker D$. Then $\nu = \rk \, D \le 3$ which is a contradiction.
\end{proof}
}
\begin{remark}
As it follows from the proof of Lemma~\ref{l:codazzi}, the algebraic statement ``a symmetric operator satisfying
\eqref{eq:codazzi} is a multiple of the identity" is valid when $2\nu < n$. In particular, when $n=16$, it remains
true, if we relax the restrictions $\nu \le 4$ of Theorem~\ref{t:cocl} to $\nu \ne 8$ (as for $n=16, \; \nu \le 8$
by \eqref{eq:radon}).
\end{remark}

Lemma~\ref{l:codazzi} implies Theorem~\ref{t:cocl} at the generic points. Indeed, by Lemma~\ref{l:codazzi}, every
$x \in M'$ has a neighbourhood $\mU$ which is either conformally flat or is conformally equivalent to a Riemannian
manifold whose curvature tensor has the form \eqref{eq:confcs}, with $\rho$ being a multiple of the identity, that is,
whose curvature tensor has a Clifford structure. It follows from \cite[Theorem~1.2]{Nhjm}, \cite[Proposition~2]{Nmm}
that $\mU$ is conformally equivalent to an open subset of one of the five model spaces: the rank-one symmetric spaces
$\bc P^{n/2}, \; \bc H^{n/2}, \; \mathbb{H}P^{n/4}, \; \mathbb{H}H^{n/4}$, or the Euclidean space.

To prove Theorem~\ref{t:cocl} in full, we show that, firstly, the same is true for any $x \in M^n$,
and secondly, that the model space to a domain of which $\mathcal{U}$ is conformally equivalent is the same,
for all $x \in M^n$.

We normalize the standard metric $\tilde g$ on each of the spaces
$\bc P^{n/2}, \; \bc H^{n/2}, \; \mathbb{H}P^{n/4}, \; \mathbb{H}H^{n/4}$ in such a way that the sectional curvature
$K_\sigma$ satisfies $|K_\sigma| \in [1,4]$. Then the curvature tensor of each of them has a Clifford structure
$\Cliff(\nu; J_1, \dots, J_\nu; \ve, \ve, \dots, \ve), \; (\nu+1 \; \ve$'s), where $\nu=1,3, \; \ve=\pm 1$ and the
$J_i$'s are smooth anticommuting almost Hermitian structures, with $J_1J_2= \pm J_3$ when $\nu=3$ and with
$\tilde \n_Z J_i = \sum_{j=1}^m \omega_{i}^{j}(Z) J_j$, where $\omega_i^j$ are smooth $1$-forms with
$\omega_i^j+\omega_j^i=0$, and $\tilde \n$ is the Levi-Civita connection for $\tilde g$. Denote the corresponding
spaces by $M_{\nu,\ve}$ (and their Weyl tensors, by $W_{\nu,\ve}$), so that
\begin{equation*}
M_{1,1}=(\bc P^{n/2},\tilde g), \quad M_{1,-1}=(\bc H^{n/2},\tilde g), \quad
M_{3,1}=(\mathbb{H}P^{n/4},\tilde g), \quad M_{3,-1}=(\mathbb{H}H^{n/4},\tilde g).
\end{equation*}

We start with the following technical lemma:
\begin{lemma}\label{l:mmeps}
Let $(N^n,\<\cdot,\cdot\>)$ be a smooth Riemannian space locally conformally equivalent to one of the $M_{\nu,\ve}$,
so that $\tilde g =f\<\cdot,\cdot\>$, for a positive smooth function $f=e^{2\phi}: N^n \to \br$.
Then the curvature tensor $R$ and the Weyl tensor $W$ of $(N^n,\<\cdot,\cdot\>)$ satisfy
\begin{subequations}\label{eq:weylconfmodel}
\begin{align}
    R(X,Y)&=(X \wedge KY + KX \wedge Y) + \ve f (X \wedge Y + T(X,Y)), \quad \text{where}\label{eq:model1} \\
    \notag
    T(X,Y)&=\sum\nolimits_{i=1}^\nu (J_iX \wedge J_i Y +2 \< J_iX, Y \> J_i), \quad
    K=H(\phi)-\n\phi \otimes \n\phi + \tfrac12 \|\nabla \phi\|^2 \id,\\
    W(X,Y)&=W_{\nu,\ve}(X,Y)=\ve f (-\tfrac{3\nu}{n-1} X \wedge Y + T(X,Y)), \label{eq:model2} \\
    \|W\|^2&= C_{\nu n} f^2, \quad C_{\nu n}=6 \nu n (n+2)(n-\nu-1) (n-1)^{-1}, \label{eq:model3} \\
    (\n_Z W)(X,Y)&=\ve Zf (-\tfrac{3\nu}{n-1} X \wedge Y + T(X,Y))  \label{eq:model4} \\
    \notag
    &+ \tfrac12 \ve([T(X,Y),\n f \wedge Z] + T((\n f \wedge Z)X,Y)+ T(X,(\n f \wedge Z)Y)),
\end{align}
\end{subequations}
where $X \wedge Y$ is the linear operator defined by $(X \wedge Y)Z=\<X,Z\>Y-\<Y,Z\>X$, $H(\phi)$ is the symmetric
operator associated to the Hessian of $\phi$, and both $\n$ and the norm are
computed with respect to $\<\cdot,\cdot\>$.
\end{lemma}
\begin{proof} The curvature tensor of $M_{\nu,\ve}$ has the form
$\tilde R(X,Y)=\ve(X \tilde \wedge Y + \sum\nolimits_{i=1}^{\nu}(J_iX \tilde \wedge J_i Y +2 \tilde g(J_iX, Y)J_i))$,
where $(X \tilde\wedge Y)Z=\tilde g(X,Z)Y-\tilde g(Y,Z)X$. Under the conformal change of metric
$\tilde g =f\<\cdot,\cdot\>=e^{2\phi}\<\cdot,\cdot\>$, the curvature tensor transforms as
$\tilde R(X,Y)=R(X,Y)-(X \wedge KY + KX \wedge Y)$. As $\tilde g(X,Y) =f\<X,Y\>$, $X \tilde\wedge Y=f (X \wedge Y)$
and the $J_i$'s remain anticommuting almost Hermitian structures for $\<\cdot,\cdot\>$, equation
\eqref{eq:model1} follows.

The fact that the Weyl tensor has the form \eqref{eq:model2} follows from the definition; the norm of $W$ can be
computed directly using the fact that the $J_i$'s are orthogonal and that $J_1J_2= \pm J_3$ when $\nu=3$.

From $\tilde \n_Z J_i = \sum_{j=1}^{\nu} \omega_{i}^{j}(Z) J_j$ and
$\tilde \n_Z X = \n_Z X +Z\phi \, X + X\phi \, Z -\<X,Z\>\phi$, where  $\tilde \n$ is the Levi-Civita connection
for $\tilde g$, we get $\n_Z J_i = \sum_{j=1}^{\nu} \omega_{i}^{j}(Z) J_j + [J_i, \n \phi \wedge Z]$ (where we used
the fact that $[J_i, X \wedge Y]=J_iX \wedge Y + X \wedge J_iY$). Then
\begin{equation*}
    (\n_Z T)(X,Y)=[T(X,Y),\n \phi \wedge Z] + T((\n \phi \wedge Z)X,Y)+ T(X,(\n \phi \wedge Z)Y),
\end{equation*}
which, together with \eqref{eq:model2}, proves \eqref{eq:model4}.
\end{proof}

For every point $x \in M'$, there exists a neighbourhood $\mathcal{U}$ of $x$ and a positive smooth function
$f: \mathcal{U} \to \br$ such that the Riemannian space $(\mathcal{U}(x), f \<\cdot,\cdot\>)$ is isometric to an open
subset of one of the five model spaces ($M_{\nu,\ve}$ or $\Rn$), so at every point $x \in M'$, the Weyl tensor $W$ of
$M^n$ either vanishes, or has the form given in \eqref{eq:model2}. The Jacobi operators associated to the
different Weyl tensors $W_{\nu,\ve}$ in \eqref{eq:model2} differ by the multiplicities and
the signs of the eigenvalues, so every point $x \in M'$ has a neighbourhood conformally equivalent to a domain of
exactly one of the model spaces. Moreover, the function $f>0$ is well-defined at all the points where $W \ne 0$,
as $\|W\|^2 = C_{\nu n} f^2$ by \eqref{eq:model3}.

By continuity, the Weyl tensor $W$ of $M^n$ either has the form $W_{\nu,\ve}$ or vanishes, at every point $x \in M^n$
(as $M'$ is open and dense in $M^n$, see Lemma~\ref{l:locc1}). Moreover,
every point $x \in M^n$, at which the Weyl tensor has the form $W_{\nu,\ve}$, has a neighbourhood, at
which the Weyl tensor has the same form. Hence $M^n=M_0 \cup \bigcup_\a M_\a$, where $M_0=\{x \, : \, W(x)=0\}$
is closed, and every $M_\a$ is a nonempty open connected subset, with $\partial M_\a \subset M_0$, such that the
Weyl tensor has the same form $W_{\nu,\ve}=W_{\nu(\a),\ve(\a)}$ at every point $x \in M_\a$. In particular,
$M_\a \subset M'$, for
every $\a$, so that each $M_\a$ is locally conformally equivalent to one of the model spaces $M_{\nu,\ve}$.

If $M = M_0$ or if $M_0=\varnothing$, the theorem is proved.
Otherwise, suppose that $M_0 \ne \varnothing$ and that there exists at least one component $M_\a$.
Let $y \in \partial M_\a \subset M_0$ and let $B_\delta(y)$ be a small geodesic ball of $M$ centered at $y$ which is
strictly geodesically convex (any two points from $B(y)$ can be connected by a unique geodesic segment lying in
$B_\delta(y)$ and that segment realizes the distance between them). 
Let $x \in B_{\delta/3}(y) \cap M_\a$ and
let $r = \mathrm{dist}(x, M_0)$. Then the geodesic ball $B=B_r(x)$ lies in $M_\a$ and is strictly convex. Moreover,
$\partial B$ contains a point $x_0 \in M_0$. Replacing $x$ by the midpoint of the segment $[xx_0]$ and $r$ by $r/2$,
if necessary, we can assume that all the points of $\partial B$, except for $x_0$, lie in $M_\a$.

The function $f$ is positive and smooth on $\overline{B} \setminus \{x_0\}$
(that is, on an open subset containing $\overline{B} \setminus \{x_0\}$, but not containing $x_0$). We are
interested in the behavior of $f(x)$, when $x \in B$ approaches $x_0$.
{
\begin{lemma}\label{l:x_0}
When $x \to x_0, \; x \in B$, both $f$ and $\nabla f$ have a finite limit. Moreover,
$\lim_{x \to x_0, x \in B}f(x)=0$.
\end{lemma}
\begin{proof}
The fact that $\lim_{x \to x_0, x \in B}f(x)=0$ follows from \eqref{eq:model3} and the fact that $W_{|x_0}=0$
(as $x_0 \in M_0$).

As the Riemannian space $(B, f\<\cdot,\cdot\>)$ is locally isometric to a rank-one symmetric space $M_{\nu,\ve}$ and
is simply connected, there exists a smooth isometric immersion $\iota:(B, f\<\cdot,\cdot\>) \to M_{\nu,\ve}$. Since
$f$ is smooth on $\overline{B} \setminus \{x_0\}$ and $\lim_{x \to x_0, x \in B}f(x)=0$, the range of $\iota$ is
a bounded domain in $M_{\nu,\ve}$. Moreover, as $\lim_{x \to x_0, x \in B}f(x)=0$, every sequence of points in $B$
converging to $x_0$ in the metric $\<\cdot,\cdot\>$ is a Cauchy sequence for the metric $f\<\cdot,\cdot\>)$. It
follows that there exists a limit $\lim_{x \to x_0, x \in B} \iota(x) \in M_{\nu,\ve}$. Defining for every
$x \in B$ the point $\cJ_{|x}=\Span(J_1,\dots,J_\nu)$ in the Grassmanian $G(\nu, \bigwedge^2 T_xM^n)$, we find that
there exists a limit $\lim_{x \to x_0, x \in B} \cJ_{|x} =:\cJ_{|x_0} \in G(\nu, \bigwedge^2 T_{x_0}M^n)$. In
particular, if $Z$ is a continuous vector field on $\overline{B}$, then there exists a unit continuous vector field
$Y$ on $\overline{B}$ such that $Y \perp Z, \cJ Z$ on $B$. For such two vector fields, the function
$\theta(Y,Z)=\<\sum_{j=1}^n (\n_{E_j} W)(E_j,Y)Y,Z\>$ (where $E_j$ is an orthonormal frame on $\overline{B}$) is
well-defined and continuous on $\overline{B}$. Using \eqref{eq:model4} we obtain by a direct computation that at the
points of $B$, $\theta(Y,Z)=\frac{\ve (n-3)}{2(n-1)} \<(3\nu \n f \wedge Y-(n-1)T(\n f,Y))Y,Z\>
=\frac{-3\ve \nu(n-3)}{2(n-1)} \<\n f ,Z\>$ (where we used the fact that $\|Y\|=1$ and $Y \perp Z, \cJ Z$). As
$\theta(Y,Z)$ is continuous on $\overline{B}$, there exists a limit $\lim_{x \to x_0, x \in B} Zf$. Since $Z$ is an
arbitrary continuous vector field on $\overline{B}$, $\n f$ has a finite limit when $x \to x_0, \; x \in B$.
\end{proof}
}


As $\lim_{x \to x_0, x \in B}f(x)=0$ and the $J_i$'s are orthogonal, the second term on the right-hand
side of equation \eqref{eq:model1} tends to $0$ when $x \to x_0$ in $B$.
Therefore the (3,1) tensor field defined by $(X,Y) \to (X \wedge KY + KX \wedge Y)$ has a finite limit
(namely $R_{|x_0}$) when $x \to x_0$ in $B$. It follows that the symmetric operator $K$ has a finite limit at $x_0$.
Computing the trace of $K$ and using the fact that $\phi = \frac12 \ln f$ we get
\begin{equation}\label{eq:laplacian}
    \triangle u = F u , \quad \text{where} \; u=f^{(n-2)/4}, \; F =\tfrac12(n-2) \Tr K
\end{equation}
on $B$. Both functions $F$ and $u$ are smooth on $\overline{B} \setminus\{x_0\}$ and have a finite limit
at $x_0$. Moreover, $\lim_{x \to x_0, x \in B}u(x)=0$ by Lemma~\ref{l:x_0} and $u(x)>0$ for
$x \in \overline{B} \setminus\{x_0\}$. The domain $B$ is a small geodesic ball, so it satisfies
the inner sphere condition (the radii of curvature of the sphere $\partial B$ are uniformly bounded).
By the boundary point theorem \cite[Section~2.3]{F}, the inner directional derivative of $u$ at $x_0$
(which exists by Lemma~\ref{l:x_0}, if we define $u(x_0)=0$ by continuity) is positive.

As $\n u=\frac14 (n-2) f^{(n-6)/4} \n f$ in $B$, we arrive at a contradiction with Lemma~\ref{l:x_0} in all the cases,
except for $n=6$. To finish the proof in that case, we will show that the limit $\lim_{x \to x_0, x \in B} \n f(x)$,
which exists by Lemma~\ref{l:x_0}, is zero. When $n=6$, we have $\nu=1$ by \eqref{eq:radon}, so
$T(X,Y)=JX \wedge J Y +2 \< JX, Y \> J$, where $J=J(x)$ is smooth on $\overline{B} \setminus\{x_0\}$ and
has a limit when $x \to x_0, \; x \in B$ (see the proof of Lemma~\ref{l:x_0}). Using the covariant derivative of $T$
computed in Lemma~\ref{l:mmeps} and \eqref{eq:model4}, we obtain that on $B$,
\begin{equation*}
\begin{aligned}
    (&\n_U\n_Z W)(X,Y)=\ve \<H(f)U,Z\> (-\tfrac{3}{5} X \wedge Y + T(X,Y))  \\
    &+ \tfrac12 \ve([T(X,Y),H(f)U \wedge Z] + T((H(f)U \wedge Z)X,Y)+ T(X,(H(f)U \wedge Z)Y)) \\
    &+ \tfrac12 \ve f^{-1} Zf ([T(X,Y),\n f \wedge U] + T((\n f \wedge U)X,Y)+ T(X,(\n f \wedge U)Y))\\
    &+ \tfrac14 \ve f^{-1} [[T(X,Y),\n f \wedge U] + T((\n f \wedge U)X,Y)+ T(X,(\n f \wedge U)Y),\n f \wedge Z]\\
    &+ \tfrac14 \ve f^{-1} ([T((\n f \wedge Z)X,Y),\n f \wedge U] + T((\n f \wedge U)(\n f \wedge Z)X,Y)+
    T((\n f \wedge Z)X,(\n f \wedge U)Y))\\
    &+ \tfrac14 \ve f^{-1} ([T(X,(\n f \wedge Z)Y),\n f \wedge U] + T((\n f \wedge U)X,(\n f \wedge Z)Y)+
    T(X,(\n f \wedge U)(\n f \wedge Z)Y)),\\
\end{aligned}
\end{equation*}
where $H(f)$ is the symmetric operator associated to the Hessian of $f$. Taking $U=Z=E_j$, where $\{E_j\}$ is an
orthonormal basis, and summing up by $j$ we find after some computations:
\begin{multline*}
\sum\nolimits_{j=1}^6 (\n_{E_j}\n_{E_j} W)(X,Y)=\ve \triangle f (-\tfrac{3}{5} X \wedge Y + T(X,Y)) -
\ve f^{-1} \|\n f\|^2 T(X,Y)\\
+ \ve f^{-1} (T(X,Y)\n f \wedge \n f + T((X \wedge Y)\n f,\n f)) 
+ \tfrac32 \ve f^{-1} (\n f \wedge (X \wedge Y)\n f  
+ J\n f \wedge (X \wedge Y)J\n f). 
\end{multline*}
As both $\n f$ and $J$ are smooth on $\overline{B} \setminus\{x_0\}$ and have limits when
$x \to x_0, \; x \in B$, there exist unit vector fields $X, Y$, continuous on $\overline{B}$ and satisfying
$\cI X, \cI Y \perp \n f, \; \cI X \perp \cI Y$. For such $X$ and $Y$,
$$
\sum\nolimits_{j=1}^6 (\n_{E_j}\n_{E_j} W)(X,Y)=\ve \triangle f (-\tfrac{3}{5} X \wedge Y + JX \wedge JY) -
\ve f^{-1} \|\n f\|^2 JX \wedge JY.
$$
As the left-hand side is continuous on $\overline{B}$ and $\lim_{x \to x_0, x \in B} \triangle f = 0$ by
\eqref{eq:laplacian} and Lemma~\ref{l:x_0}, we obtain that the field $f^{-1} \|\n f\|^2 JX \wedge JY$ of
skew-symmetric operators has a limit at $x_0$. Taking the trace of its square we find that there exists a limit
$\lim_{x \to x_0, x \in B} f^{-2} \|\n f\|^4$ which implies
$\lim_{x \to x_0, x \in B} \n f=0$ by Lemma~\ref{l:x_0}. We again arrive at a contradiction with the
boundary point theorem for the function $u=f$ satisfying \eqref{eq:laplacian}.
\end{proof}


\begin{thebibliography}{BGNSt}

\bibitem[ABS]{ABS}
Atiah M.F., Bott R., Shapiro A.
\emph{Clifford modules},
Topology, \textbf{3, suppl.1} (1964), 3 -- 38.

\bibitem[BG1]{BG1}
Bla\v zi\'c N., Gilkey P.
\emph{Conformally Osserman manifolds and conformally complex space forms},
Int. J. Geom. Methods Mod. Phys. \textbf{1} (2004), 97 -- 106.

\bibitem[BG2]{BG2}
Bla\v zi\'c N., Gilkey P.
\emph{Conformally Osserman manifolds and self-duality in Riemannian geometry}.
Differential geometry and its applications, 15--18, Matfyzpress, Prague, 2005.

\bibitem[BGNSi]{BGNSi}
Bla\v zi\'c N., Gilkey P., Nik\v cevi\'c S., Simon U.
\emph{The spectral geometry of the Weyl conformal tensor}.
PDEs, submanifolds and affine differential geometry, 195--203, Banach Center Publ.,
\textbf{69}, Polish Acad. Sci., Warsaw, 2005.

\bibitem[BGNSt]{BGNSt}
Bla\v zi\'c N., Gilkey P., Nik\v cevi\'c S., Stavrov I.
\emph{Curvature structure of self-dual $4$-manifolds},
arXiv: math.DG/ 0808.2799.

\bibitem[Chi]{Chi}
Chi Q.-S.
\emph{A curvature characterization of certain locally rank-one symmetric spaces},
J. Differ. Geom. \textbf{28}(1988), 187 -- 202.

\bibitem[DS]{DS}
Derdzinski A., Shen C.-L.
\emph{Codazzi tensor fields, curvature and Pontryagin forms},
Proc. London Math. Soc. (3), \textbf{47} (1983), 15 -- 26.

\bibitem[F]{F}
Fraenkel L. E.
\emph{An introduction to maximum principles and symmetry in elliptic problems},
Cambridge Tracts in Mathematics, 128. Cambridge University Press, Cambridge, 2000.

\bibitem[GKV]{GKV}
García-R\'{\i}o E., Kupeli D., V\'{a}zquez-Lorenzo R.
\emph{Osserman manifolds in semi-Riemannian geometry}.
Lecture Notes in Mathematics, 1777. Springer-Verlag, Berlin, 2002.

\bibitem[GSV]{GSV}
Gilkey P., Swann A., Vanhecke L.
\emph{Isoparametric geodesic spheres and a conjecture of Osserman concerning the Jacobi operator},
Quart. J. Math. Oxford (2), \textbf{46}(1995), 299 -- 320.

\bibitem[G1]{G1}
Gilkey P.
\emph{Geometric properties of natural operators defined by the Riemann curvature tensor}.
World Scientific Publishing Co., Inc., River Edge, NJ, 2001.

\bibitem[G2]{G2}
Gilkey P.
\emph{The geometry of curvature homogeneous pseudo-Riemannian manifolds}.
ICP Advanced Texts in Mathematics, 2. Imperial College Press, London, 2007.

\bibitem[HL]{HL}
Harvey R., Lawson H.B, \emph{Calibrated geometries}, Acta Math. \textbf{148} (1982), 47--157.

\bibitem[H]{H}
D.Husemoller, \emph{Fiber bundles}, (1975), Springer-Verlag.

\bibitem[LM]{LM}
H.B.Lawson, M.-L.Michelsohn, \emph{Spin geometry}, (1989), Princeton Univ. Press.


\bibitem[N1]{Nhjm}
Nikolayevsky Y.
\emph{Osserman manifolds and Clifford structures},
Houston J. Math. \textbf{29}(2003), 59--75.

\bibitem[N2]{Nmm}
Nikolayevsky Y.
\emph{Osserman manifolds of dimension $8$},
Manuscripta Math. \textbf{115}(2004), 31 -- 53.

\bibitem[N3]{Nma}
Nikolayevsky Y.
\emph{Osserman Conjecture in dimension $n \ne 8, 16$},
Math. Ann. \textbf{331}(2005), 505 -- 522.

\bibitem[N4]{Nbel}
Nikolayevsky Y.
\emph{On Osserman manifolds of dimension $16$},
Contemporary Geometry and Related Topics, Proc. Conf. Belgrade, 2005 (2006),
379 -- 398.

\bibitem[Ol]{Ol}
Olszak Z.
\emph{On the existence of generalized space forms},
Israel J. Math.
\textbf{65}(1989), 214 -- 218.

\bibitem[Os]{O}
Osserman R.
\emph{Curvature in the eighties},
Amer. Math. Monthly,
\textbf{97}(1990), 731 -- 756.

\bibitem[Pf]{Pf}
Pfister A.
\emph{Quadratic forms with applications to algebraic geometry and topology},
London Math. Soc. Lecture Notes Ser., \textbf{217}, (1995),
Cambridge Univ. Press.

\end{thebibliography}
\end{document}